\documentclass{article}

\usepackage{graphicx,psfrag}
\usepackage{amssymb,amsmath,amsthm}
\usepackage{enumerate}  
\usepackage{a4wide}
\usepackage{blindtext}
\usepackage{caption}
\usepackage{subcaption}

\usepackage{mathrsfs}
\usepackage{enumitem}

\numberwithin{equation}{section}
\allowdisplaybreaks[4]

\newtheorem{proposition}{Proposition}[section]
\newtheorem{theorem}[proposition]{Theorem}
\newtheorem{lemma}[proposition]{Lemma}
\newtheorem{corollary}[proposition]{Corollary}
\newtheorem{definition}[proposition]{Definition}
\theoremstyle{definition}
\newtheorem{remark}[proposition]{Remark}
\newtheorem{example}[proposition]{Example}

\newcommand{\N}{\mathbb{N}}     
\newcommand{\R}{\mathbb{R}}     
 
\def\r{\R}

\newcommand{\calA}{\mathscr{A}}

\newcommand{\calC}{\mathscr{C}}

\newcommand{\calK}{\mathscr{K}}

\newcommand{\calO}{\mathscr{O}}

\DeclareMathOperator{\cl}{cl}				

\newcommand{\ox}{\calO(X)}
\newcommand{\cx}{\calC(X)}

\newcommand{\ax}{\calA(X)}

\newcommand{\ocx}{\calO_{c}(X)}
\newcommand{\ccx}{\calC_{c}(X)}

\newcommand{\osx}{\calO_{s}(X)}
\newcommand{\csx}{\calC_{s}(X)}
\newcommand{\asx}{\calA_{s}(X)}
\newcommand{\ossx}{\calO_{ss}(X)}

\newcommand{\assx}{\calA_{ss}(X)}

\newcommand{\bossx}{\calO_{ss}^{*}(X)}

\newcommand{\ksx}{\calK_{s}(X)}
\newcommand{\bosx}{\calO_{s}^{*}(X)}
\newcommand{\bcox}{\calK_{0}(X)}
\newcommand{\kox}{\calK_{0}(X)}

\newcommand{\cox}{\calC_{0}(X)}
\newcommand{\oox}{\calO_{0}(X)}
\newcommand{\aox}{\calA_{0}(X)}

\renewcommand{\O}{\emptyset}

\def\sm{\setminus}
\def\cl{\overline}
\def\ol{\overline}
\def\se{\subseteq}
\def\sc{\sqcup}
\def\bsc{\bigsqcup}
\def\bc{\bigcup}

\def\eps{\epsilon}

\def\la1{\lambda_1}
\def\la2{\lambda_2}
\def\la0{\lambda_{0}}
\def\la{\lambda}

\begin{document}

\title{Semisolid sets and topological measures}
\author{Svetlana V. Butler\footnote{Department of Mathematics,  University of California, Santa Barbara,  552 University Rd, Isla Vista, CA 93117, USA. 
Email: svtbutler@ucsb.edu, svetbutler@gmail.com}
}
\maketitle

\maketitle
\begin{abstract}
This paper is one in a series that investigates topological measures 
on locally compact spaces. A topological measure is a set function which is finitely additive on the collection of open and compact sets, inner regular on open sets, and outer regular on closed sets. We examine semisolid sets and give a way of constructing topological measures from solid-set functions on locally compact, connected, locally connected spaces. 
For compact spaces our approach produces a simpler method than the current one. 
We give examples of finite and infinite topological measures on 
locally compact spaces and present an easy way to generate topological measures on
spaces  whose one-point compactification has genus 0.
Results of this paper are necessary for various methods for constructing topological measures, give additional properties of topological measures, and provide a tool
for determining whether two topological measures or quasi-linear functionals are the same.  
\end{abstract}




\section{Introduction}

This paper deal with the theory of topological measures and quasi-linear functionals.  
The origins of the theory are connected to quantum physics 
and have a fascinating history. 
Mathematical interpretations of quantum physics by G. W. Mackey and R. V. Kadison  (\cite{MackeyPaper}, \cite{MackeyBook}, \cite{Kadison}) 
led to very interesting mathematical problems. 
Let $ R$ be a von Neumann algebra, 
and let $P$ be the lattice of orthogonal projections in $R$. 
in $P$. 
The extension problem 
asks whether given a measure $\mu$ on $P$ there exists a positive state $ \rho$ on $R$ such that
$\rho | P = \mu$. A. Gleason (\cite{Gleason}) obtained the first important affirmative answer when 
$R $ is the family of all bounded linear operators on a separable Hilbert space $H$ with $dim\,  H \ge 3$.
The extension problem 
can be viewed as a special case of the following linearity problem for quasi-states 
(see \cite{Aarnes:PhysStates69},  \cite{Aarnes:QuasiStates70}).
Let $ \mathscr{A}$ be a $C^*-$algebra with identity $1$. A quasi-state is a function $\rho:  \mathscr{A} \longrightarrow \mathbb{C} $ 
which is a state on each  $C^*-$subagebra of $ \mathscr{A}$ 
generated by a single self-adjoint $a \in \mathscr{A}$ and $1$, and which satisfies $ \rho(a + ib) = \rho(a) + i \rho(b)$ for  self-adjoint $a,b \in  \mathscr{A}$. 
The problem is to determine whether $\rho$ is linear. 
In \cite[Theorem 1]{Aarnes:PhysStates69} J. Aarnes claimed that any positive quasi-linear functional $ \rho$ on an abelian 
$C^*-$algebra  is linear.  However, C. Akemann and M. Newberger found a gap in the proof (see \cite{Akemann}). It turned out that the gap was
unbridgeable, 
which J. Aarnes demonstrated almost twenty years later.
Any abelian unital $C^*$-algebra is isomorphic to $C(X)$ for some compact Hausdorff space $X$, and in his seminal paper  \cite{Aarnes:TheFirstPaper}  
Aarnes introduced  set functions generalizing measures 
(initially called quasi-measures, now topological measures)
and corresponding quasi-linear functionals on $C(X) $ for a compact Hausdorff space. 
Quasi-linear functionals (also called quasi-integrals) are functionals that are linear on singly generated subalgebras, but in general not linear. 
On locally compact spaces, there is an order-preserving bijection between quasi-linear functionals 
and compact-finite topological measures, which is also "isometric" when topological measures are finite (\cite{Aarnes:TheFirstPaper}, 
 \cite{Butler:QLFLC}, \cite{Butler:ReprDTM}).

Interestingly, quasi-linear functionals are also related to the mathematical model of quantum mechanics of von Neumann (\cite{vonN}).
Let $ \mathscr{A}$ be the algebra of observables in quantum mechanics. In a simple form, $ \mathscr{A}$ is a space of Hermitian operators 
on a finite dimensional Hilbert space. In von Neumann's definition, a quantum state is a linear positive normalized functional on $ \mathscr{A}$.  
A number of physicists disagreed with the additivity axiom $\rho(A+ B) = \rho(A) + \rho(B)$, 
arguing that it makes sense only if observables $A$ and $B$ are 
simultaneously measurable, i.e commute. Two Hermitian operators on a finite dimensional Hilbert space commute iff 
they can be written as the polynomials of the same Hermitian operator.
This justifies the following modifications of the additivity axiom: 
($*$)  $\rho(A+ B) = \rho(A) + \rho(B)$ if $A, B$ belong to a singly generated subalgebra of $ \mathscr{A}$. 
($**$) $\rho(A+ B) = \rho(A) + \rho(B)$ if $[A,B]_{\hbar} = 0$, where the bracket $[A,B]_{\hbar} = -\frac{i}{ \hbar} (AB - BA)$, 
and $\hbar$ is the Planck constant. 
A positive homogeneous functional with additivity as in ($*$) is a positive quasi-linear functional introduced by Aarnes, while  
 additivity as in ($**$) leads to the notion of a Lie quasi-state (see \cite{EPLie}, \cite[Sect. 5.6]{PoltRosenBook}).
For more information about the physical interpretation of quasi-linear functionals 
see \cite{Aarnes:PhysStates69}, \cite{Aarnes:QuasiStates70}, \cite{Aarnes:TheFirstPaper}, 
\cite{Entov},  \cite{EntovPolterovich}, \cite{EPZ},  \cite{EPZphys}, \cite{PoltRosenBook}.
 
M. Entov and L. Polterovich first linked the theory of quasi-linear functionals and topological measures to symplectic geometry. 
Their seminal paper \cite{EntovPolterovich} was followed by extensive research  in which 
topological measures and quasi-linear functionals are used in connection with rigidity phenomenon in symplectic geometry. 
Topological measures and quasi-linear functionals became an indispensable part of function theory on symplectic manifolds, 
which is the subject of an excellent monograph \cite{PoltRosenBook}.  
 
Entov and Polterovich introduced symplectic quasi-states and partial symplectic quasi-states, 
which are subclasses of quasi-linear functionals
(see \cite{EntovPolterovich}, \cite{Entov}, \cite{PoltRosenBook}). 
Symplectic quasi-states exist on a variety of manifolds, including $\mathbb{C} P^n$, complex Grassmanian, $S^2$, 
$S^2 \times S^2$ (see  \cite{Entov}, \cite{EntovPolterovich}, \cite{EntovPolterovich08}, \cite[Ch. 5]{PoltRosenBook}).
On a closed oriented surface any positive quasi-linear functional is a symplectic quasi-state (see \cite[Ch. 5]{PoltRosenBook}).  
Symplectic quasi-states and topological measures are closely related to and can be obtained by homogeneous quasi-morphisms. 
One can also determine that a quasi-morphism is not a homomorphism by showing that a certain quasi-state is not linear, 
i.e. a certain topological measure is not a measure (see \cite{Lanzat}, for example). 
Topological measures can be used to 
distinguish Lagrangian knots that have identical classical invariants (\cite[Sect. 6.2, Sect. 12.6]{PoltRosenBook}).
Symplectic quasi-states help to determine how well a pair of functions can be approximated by a pair of Poisson commuting functions, 
and, more generally, provide bounds for the profile function ( \cite{BEP}, \cite[Sect. 4.3]{Entov}, \cite{EPZ}, \cite[Sect. 8.3]{PoltRosenBook}). 
Symplectic quasi-states, unlike Langrangian Floer theory, allow one to prove results about nondisplaceability 
for singular sets (\cite[Sect. 4.1, Sect. 4.5]{Entov}). 
A nice account of applications of topological measures and symplectic quasi-states to symplectic geometry can be found 
in \cite[Ch.6]{PoltRosenBook}. 
  
For a (symplectic) quasi-state $\rho$ we may define $ \pi(f,g) =  |\rho(f+g) - \rho(f) - \rho(g) |$. 
The functional $\pi$ (which is nontrivial iff $ \rho$ is not linear, i.e. the corresponding topological measure is not a measure) 
is important in a number of interesting results in symplectic geometry. 
The Poisson bracket $\{f,g\}$ of functions $f,g$ involves first derivatives of the functions, 
and at first glance there is no restriction on change in the uniform norm of  $\{f,g\}$ resulting from perturbations in $f,g$. 
For a symplectic quasi-state $\rho$ there is an estimate  $ \pi(f,g)  \le const  \, \sqrt{ \| \{f,g \} \|}$, and, therefore, 
there is such a restriction (see \cite{EntPoltRosen}, \cite{EPZ}). For this remarkable phenomenon
as well as for a discussion of how $ \pi(f,g) $ appears in the context of simultaneous measurement of noncommuting observables $f,g$ and provides 
a lower bound for the error,  see \cite{EPZ}.  
 
Function theory on symplectic manifolds would not be possible without (a) the theory of topological measures and quasi-linear functionals and 
(b) rigidity phenomenon.   
The $C^0$-rigidity property holds for open or closed manifolds, and it allows one to extend the notion of Poisson commutativity from smooth 
to continuous functions (see \cite{EPRigid}, \cite[Sect. 2.1]{PoltRosenBook}). 
At the moment, function theory on symplectic manifolds is mostly developed for closed manifolds, perhaps because until very recently
the theory of topological measures  and  quasi-linear functionals dealt almost exclusively with the compact case.
We believe that results of this paper 
(together with other recent papers devoted to the theory of topological measures and quasi-linear functionas on locally compact spaces) 
may allow 
extension of the fascinating function theory on symplectic manifolds to nonclosed manifolds and lead the way to
new contributions.

The theory of quasi-linear functionals is connected with the theory of Choquet integrals.  
If $\mu$ is a topological measure, the quasi-linear functional $ \int_X f \, d\mu$  is a symmetric Choquet integral 
(see \cite{Butler:QLFLC}, \cite{Butler:ReprDTM} and \cite[Ch. 7]{DD}).
Many results about Choquet integrals are for a supermodular (also called 2-monotone) or totally monotone set function, 
and/or for a set function whose domain is a $\sigma$-algebra, an algebra, or is closed under intersection and union.  
None of this is applicable for topological measures, so the results of Choquet theory do not automatically translate 
for quasi-linear functionals. 
Nevertheless, we do have some of the typical results of Choquet integrals, such as properties of being  monotone, homogeneous, and
additive on comonotonic functions, and they hold 
sometimes under weaker conditions. 
Some results for quasi-linear functionals are stronger than those for Choquet integrals; others 
show that some assumptions can not be weakened in Choquet theory results. 
The interconnection of the  two theories deserves more investigation.  
This will require deeper understanding of topological measures, of which this paper is a part. 

Topological measures and deficient topological measures are defined on open and closed subsets of a topological space, 
which means that there is no algebraic structure on the  domain.  
They lack subadditivity and other properties typical for measures, and many standard techniques 
of measure theory and functional analysis do not apply to them. 
Nevertheless, many classical results of probability theory hold for topological measures. 
These include Aleksandrov's Theorem for equivalent definitions of weak convergence of  topological and deficient topological measures. 
There is also a version of Prokhorov's Theorem, which relates the existence of a weakly convergent subsequence in any sequence
in a family of topological measures  
to the characteristics of being a uniformly bounded in variation and uniformly tight family.  
It is also possible to define Prokhorov and Kantorovich-Rubenstein metrics and show that convergence in either of them   
implies weak convergence of topological measures.  
See \cite{Butler:WkConv}.

There are connections of the theory of quasi-linear functionals and topological measures with symplectic geometry, probability theory,
Choquet theory, fractals, etc. 
Current results are just scratching the surface of interesting and fruitful investigations in different directions. 
Until very recently,
other than paper  \cite{Alf:ReprTh} and a couple of preprints, including \cite{Aarnes:LC}, 
there were no works devoted to quasi-linear functionals and topological measures on locally compact spaces. 
This 
impeded both the development of the field and its connection with other areas of mathematics. 
To remedy the situation, the author has written several papers extending the theory to the locally compact setting. 
The current paper, devoted to semisolid sets and construction of topological measures on locally compact spaces, is a key part of that series.
This paper is for
anyone interested in (1) learning about quasi-linear functionals  on locally compact 
noncompact spaces or on compact spaces; (2) further study 
of quasi-linear functionals, signed quasi-linear functionals, and other related nonlinear functionals; 
(3) applying topological measures and quasi-linear functionals in other areas of mathematics.  

Aarnes's fundamental paper  \cite{Aarnes:ConstructionPaper}, devoted to construction of topological measures on compact spaces, is
perhaps the most technically difficult of all papers on topological measures in the compact case. 
The construction technique employed 
in \cite{Aarnes:ConstructionPaper} 
was later nicely simplified by D. Grubb, but, unfortunately, the simplification was never published.  
Transitioning to the locally compact situation is not mechanical. (For one thing, 
on a compact space we work with open sets and closed sets, while on a locally compact space 
the focus is on compact sets and open sets, which are no longer complements of each other. One also has to deal with unbounded sets.) 
A  preprint by Aarnes \cite{Aarnes:LC} is devoted to extension of the results from \cite{Aarnes:ConstructionPaper} 
to the locally compact case.
While this work contains many excellent ideas, it is not entirely satisfactory. It is very technical, long,
sometimes proves only parts of stated theorems, at times askes the reader to adapt lengthy proofs from other papers to its subject matter,
and does not quite do what is needed (see, for instance,  \cite[Sect. 6]{Aarnes:LC} 
where examples are obtained by a method which in general can produce a trivial topological measure from a non-trivial initial set function). 
It has also never been published in a mainstream journal and has remained a hard to obtain and even harder 
to understand preprint. 

In this paper we (I) introduce a concept of semisolid sets on locally compact spaces 
and study the structure of solid and semisolid sets 
(II) develop an approach for constructing topological measures on Hausdorff locally compact spaces. 
Our results allows us to extend a solid-set function 
to a topological measure on $X$ when $X$ is a Hausdorff, locally compact, connected, and locally connected space. 
The restriction of a topological measure to solid sets with compact closure is a solid-set function that uniquely determines the topological measure. 
We obtain an easy way to construct topological measures on noncompact locally compact spaces whose one-point compactification has genus 0. 
Thus, we are able to produce a variety of finite and infinite topological measures on $\mathbb{R}^n$, half-spaces, punctured balls, and so on.  
When $X$ is compact our approach produces a simpler method for constructing a topological measure from a solid-set function than the one in 
\cite{Aarnes:ConstructionPaper}  (the only method currently available).
Results of this paper are at the core of obtaining various methods for constructing topological measures and quasi-linear functionals; 
they are essential for studying their properties and also immediately provide some properties of topological measures; they give an effective method 
for determining whether two topological measures or quasi-linear functionals are the same.  

The paper is organized as follows. In Section \ref{Prelim} we give necessary topological preliminaries. 
In Section \ref{SolidSemisolid} we define semisolid and solid sets and study solid hulls of connected sets. 
In Section \ref{Sstructure} we study the structure of solid and semisolid sets. 
In Section \ref{TM} we give a definition and basic properties of topological measures on locally compact spaces, and in Section \ref{SSF} 
we do the same for solid-set functions. In Section \ref{ExtBssKc} on a locally compact, connected, and locally connected space
we extend a solid-set function from bounded solid sets to compact connected and bounded semisolid sets. 
In Section \ref{BCOX} the extension is done to the 
finite unions of disjoint compact connected sets, and in Section \ref{ExttoTM} to open and closed sets.  
In Section \ref{finAdAll} we show that extension produces a topological measure that is uniquely defined by 
a solid-set function (see Theorem \ref{extThLC} and Theorem \ref{ExtUniq}). 
In Section \ref{IrrPartInfo} we discuss irreducible partitions and genus of a space.
In Section \ref{SSFAarnes} we define a solid-set function on a compact space in a different way 
and in Section \ref{Dan'sExtension} we extend it to a topological measure. 
In Section \ref{ExamplesC} we give examples of topological measures on compact spaces.
In Section \ref{ExamplesLC} we give examples and present 
an easy way (Theorem \ref{tmXtoXha}) to generate topological measures on  Hausdorff, locally compact, connected, and locally connected spaces  
whose one-point compactification has genus 0.


The spaces we consider in this paper are Hausdorff. 

In this paper by a component of a set we always mean a connected component. We denote by $\overline E$ the closure of a set $E$ 
and by $\partial E$ the boundary of $E$. We denote by $ \bigsqcup$ a union of disjoint sets, and by $|J|$ the cardinality of a finite set $J$. 
A set $E$ is co-connected if its complement is connected. 
A set $A \subseteq X$ is called bounded if $\overline A$ is compact. 

Several collections of sets will be used often.   They include:
$\mathscr{O}(X)$,  the collection of open subsets of   $X $; $\mathscr{C}(X)$,  the collection of closed subsets of   $X $; and
$\mathscr{K}(X)$,  the collection of compact subsets of   $X $. Let $\mathscr{A}(X) = \mathscr{C}(X) \cup \mathscr{O}(X)$.
If $X$ is compact, of course, $\mathscr{C}(X) = \mathscr{K}(X)$. 

Often we will work with open, compact or closed sets with particular properties.
We use subscripts $c, s$ or $ ss$ to indicate (open, compact, closed) sets that are, respectively,  connected, solid, or semisolid. For example,
$\mathscr{O}_c(X)$   is the collection of open connected subsets of  $X$, and 
$ \mathscr{K}_s(X)$  is the collection of compact solid subsets of  $X$. 

Given any collection $\mathscr{E}$ of subsets of $X$ we denote by $\mathscr{E}^*$ the subcollection of all bounded sets belonging to $\mathscr{E}$. 
For example,  
$\mathscr{A}^{*}(X) = \mathscr{K}(X) \cup \mathscr{O}^{*}(X)$ is the collection of compact and bounded open sets, and 
$\mathscr{A}_{ss}^{*}(X) = \mathscr{K}_{ss}(X) \cup \mathscr{O}_{ss}^{*}(X) $ is the collection of  compact semisolid  and bounded open semisolid  
sets. By $\mathscr{K}_{0}(X)$ we denote the collection of finite unions of disjoint compact connected sets.

\begin{definition}
A nonnegative set function $ \mu$ on a family of sets that includes compact sets is called compact-finite if $  \mu(K) <\infty$ for each compact $K$.
A set function $ \mu$ is monotone on a collection of sets  $\mathscr{E}$ if $ \mu(A) \le \mu(B)$ whenever $ A \subseteq B, A, B \in \mathscr{E}$.
A nonnegative set function $\mu$ is called simple if it only assumes values $0$ and $1$; $\mu$ on $X$ is finite if $\mu(X) < \infty$.
\end{definition}

We consider set functions that are not identically $\infty$.

\section{Preliminaries} \label{Prelim}

This section contains necessary topological preliminaries. 
Some results in this section are known, but sometimes we give proofs for the reader's convenience.  

\begin{remark} \label{netsSETS}
An application of compactness (see, for example, \cite[Corollary 3.1.5]{Engelking})  shows that \\
(i)
If $K_\alpha \searrow K, K \subseteq U,$ where $U \in \mathscr{O}(X),\  K, K_\alpha \in \mathscr{C}(X)$, 
and $K$ and at least one of $K_\alpha$ are compact, 
then there exists $\alpha_0$ such that $ K_\alpha \subseteq U$ for all $\alpha \ge \alpha_0$. \\
(ii)
If $U_\alpha \nearrow U, K \subseteq U,$ where $K \in \mathscr{K}(X), \ U, \ U_\alpha \in \mathscr{O}(X)$ then there exists $\alpha_0$ such that
$ K \subseteq U_\alpha$ for all $\alpha \ge \alpha_0$.
\end{remark}

\begin{remark} \label{OpenComp}
(a) Suppose $X$ is connected,  $ U \in \mathscr{O}_{c}(X)$ and $F \in \mathscr{C}_{c}(X)$ are disjoint sets. 
If $\overline U \cap F \ne \emptyset$ then $U \sqcup F $ is connected. \\
(b) If $X$ is locally compact and locally connected, for each $x  \in X$ and each open set $U$ containing $x$, there is
a connected open set $V$ such that $ x \in V \subseteq \overline V \subseteq U $ and $\overline V$ is compact.\\
(c) If $V = \bigsqcup_{t \in T} V_t$ where $V$ and $V_t $ are open sets, then $\overline{ V_t} \cap V_r = \emptyset$ for $ t \ne r$. 
In particular, if $X$ is locally connected, and $V = \bigsqcup_{t \in T} V_t$  is a decomposition of an open set $V$  into connected components,  
then all components $V_t$ are open, and $\overline{ V_t} \cap V_r = \emptyset$ for $ t \ne r$. \\
(d) If $ \{ E_t: \, t \in T \}$ is a family of connected sets such that any two of the sets have a nonempty intersection then $ \bc_{t \in T} E_t$ is connected. 

\end{remark}

\begin{lemma} \label{prelLemma}
Let $U$ be an open connected subset of a locally compact  and locally connected set $X$. 
Then for any $x, y \in U$ there is $ V_{xy}  \in \mathscr{O}_{c}^{*}(X)$ such that $ x,y \in V_{xy} \subseteq \overline{ V_{xy}} \subseteq U$.
\end{lemma}

\begin{proof}
Fix $x \in U$. Let 
$A = \{ y \in U :  \exists V_{xy}  \in \mathscr{O}_{c}^{*}(X) \mbox{  such that  } 
x,y \in V_{xy} \subseteq \overline{ V_{xy}} \subseteq U \}.  $
Clearly, $A$ is open.
The set $ U \setminus A$ is also open: if $y \in U \setminus A$ pick
by Remark \ref{OpenComp} $V \in \mathscr{O}_{c}^{*}(X)$ such that  $y \in V \subseteq \overline V \subseteq U$,
then $V \subseteq U \setminus A$. 
Since $A,  U \setminus A $ are open, and  $x  \in A$, we have $A = U$.
\end{proof}

\begin{lemma} \label{easyLeLC}
Let $K \subseteq U, \ K \in \mathscr{K}(X),  \ U \in \mathscr{O}(X)$ in a locally compact space $X$.
Then there exists a set  $V \in \mathscr{O}^{*}(X)$ such that 
$ K \subseteq V \subseteq \overline V \subseteq U. $ (See, for example, \cite[Chapter XI, 6.2]{Dugundji})
\end{lemma}


\begin{lemma} \label{LeConLC}
Let $X$ be a locally compact, locally connected space, $K \subseteq U$, $K \in \mathscr{K}(X)$,  $U \in \mathscr{O}(X)$.
If either $K$ or $U$ is connected there exist  $V \in \mathscr{O}_{c}^{*}(X)$  and $C \in \mathscr{K}_{c}(X)$ such that 
$ K \subseteq V \subseteq C \subseteq U. $ 
One may take $C = \overline V$.
\end{lemma}

\begin{proof}
Case 1:  $ K \in \mathscr{K}_{c}(X)$.   For each $ x \in K$ 
by Remark \ref{OpenComp} there is $ V_x \in \mathscr{O}_{c}^{*}(X)$ such that
$ x \in V_x \subseteq \overline{ V_x} \subseteq U$. By compactness of $K$, we write 
$ K \subseteq V_{x_1} \cup \ldots \cup  V_{x_n}$.
Since 
$ x_i \in K \cap  V_{x_i}$, 
$\, K \cup V_{x_i}$ is connected for each $ i =1, \ldots, n$.
Hence, 
$ V = \bigcup_{i=1}^n V_{x_i} =  \bigcup_{i=1}^n (K \cup V_{x_i}) $
is a bounded open connected set for which 
$ K \subseteq V \subseteq \overline V \subseteq \bigcup_{i=1}^n \overline{ V_{x_i}} \subseteq U. $
Take $C = \overline V$. \\
Case 2: $ U  \in \mathscr{O}_{c}(X)$. 
As in Case 1 we may find   $V_1, \ldots, V_n \in \mathscr{O}_{c}^{*}(X)$ such that
$ K \subseteq V_1 \cup  \ldots  \cup V_n \subseteq \overline{ V_1} \ldots \cup \overline{V_n} \subseteq U .$
Pick $ x_i \in V_i$ for  $i=1, \ldots, n$.  By Lemma \ref{prelLemma} choose
$W_i \in \mathscr{O}_{c}^{*}(X)$ with $ x_1, \ x_i \in W_i \subseteq \overline{ W_i}  \subseteq U$ for $ i=2, \ldots, n$.
The set $V_1 \cup W_i \cup V_i$ is connected for each $ i =2, \ldots, n$. Then 
$ V = \bigcup_{i=1}^n V_i \cup \bigcup_{i=2}^n W_j = \bigcup_{i=2}^n (V_1 \cup W_i \cup V_i)$
is  open connected and 
$ K \subseteq \bigcup_{i=1}^n V_i \subseteq V \subseteq \overline V \subseteq \bigcup_{i=1}^n \overline{ V_i} \cup 
\bigcup_{i=2}^n \overline{ W_i} \subseteq U.$
Again, let $C = \overline V$.
\end{proof}

\begin{lemma} \label{LeCCoU}
Let $X$ be a locally compact, locally connected space. Suppose $K \subseteq U, \ K \in \mathscr{K}(X), \ U \in \mathscr{O}(X)$. 
Then there exists $C \in \mathscr{K}_{0}(X)$ such that $ K \subseteq C \subseteq U$.
\end{lemma}

\begin{proof}
Let $U = \bigsqcup_{i \in I'} U_i$ be the decomposition  into connected components. Since
$X$ is locally connected, each $U_i$ is open, and by compactness of $K$ there exists  a finite set $I \subseteq I'$ such that $K \subseteq \bigsqcup_{i \in I} U_i$.  
Then $ K \cap U_i  = K \setminus  \bigsqcup_{j \in I, \ j \ne i} U_j$ is a compact set. For each $i \in I$ by Lemma \ref{LeConLC} 
choose $C_i \in \mathscr{K}_{c}(X)$ such that $K \cap U_i \subseteq C_i \subseteq U_i$. The set $C = \bigsqcup_{i \in I} C_i$ is the desired set. 
\end{proof}

\begin{lemma} \label{CmpCmplBdA}
Let $X$ be a connected, locally connected space.
Let $ A \in \mathscr{A}_{c}(X)$ and 
let $B$ be a component of $X \setminus A$. Then
\begin{itemize}
\item[(i)]
If $A$ is open then $B$ is closed and $\overline{A} \cap B  \neq \emptyset$. 
\item[(ii)]
If $A$ is closed then $B$ is open and  $A \cap \overline{B} \neq \emptyset.$ 
\item[(iii)]
$A \sqcup \bigsqcup_{s \in S} B_s$ is connected for 
any family $\{ B_s\}_{s \in S} $ of components of $ X \setminus A$.
\item[(iv)]  
$B$ is connected and co-connected.
\end{itemize}
\end{lemma}

\begin{proof}
The proof of (i) and (ii) is not hard. 
(iii) Apply parts (a) and (d) of Remark \ref{OpenComp}. 
(iv) Let $X \setminus A = \bigsqcup_{s \in S}  B_s$ be a decomposition 
into connected components. For each $ t \in S$ 
the set $ X \setminus B_t =A  \sqcup \bigsqcup_{s \ne t}  B_s $  is a connected set  by the previous part. 
\end{proof} 

\begin{lemma} \label{LeAaCompInU}
Let $X$ be a connected, locally connected space.
Let $K \in \mathscr{K}(X), \ K \subseteq U \in \mathscr{O}_{c}^{*}(X)$. Then at most a finite number of connected components of 
$X \setminus$ K are not contained in $U.$
\end{lemma}

\begin{proof}
Let $ X  \setminus K = \bigsqcup_{s \in S} W_s$  be the decomposition of $ X \setminus K$ into connected components.
Each component $ W_s$ intersects $U$ since otherwise we would have 
$W_s \subseteq X \setminus U$, so $\overline{ W_s} \subseteq X \setminus U$, and $\overline{ W_s} \cap K =\emptyset$, 
which contradicts Lemma \ref{CmpCmplBdA}.
Assume that there are infinitely many components of $X \setminus K$ that are not contained in $U$.
Then we may choose components $W_i, \ i=1, 2, \ldots$,  such that $W_i \cap U \neq \emptyset$ and 
$W_i \cap (X \setminus U) \neq \emptyset$ for each $i$. 
Connectedness of $W_i$ implies that $ W_i \cap \partial U \neq \emptyset$ for each $i$.
Let $x_i \in W_i \cap \partial U$. By compactness of $\partial U$, let $x_0 \in \partial U$ 
be the limit point of $(x_i)$. Then $x_0  \in X \setminus U \subseteq X \setminus K = \bigsqcup_{s \in S} W_s$, 
i.e. $x_0 \in W_t $ for some $t \in S$. But then all but finitely many $x_i$ must also be in $W_t$, which is impossible, 
since $W_i \cap W_t =\emptyset$ for $t \neq i$. 
\end{proof}

\begin{corollary} \label{CoBddComp}
Let $X$ be a connected, locally connected space.
Let $K \in \mathscr{K}(X)$ and let $W$ be the union of bounded components of $X \setminus K$. 
Then $W \in \mathscr{O}^{*}(X)$.
\end{corollary}

\begin{proof}
By Lemma \ref{LeConLC} pick $V \in \mathscr{O}_{c}^{*}(X)$ such that 
$ K \subseteq V$. From Lemma  \ref{LeAaCompInU} it follows that $W$ is bounded. By Lemma 
\ref{CmpCmplBdA} $W$ is open.
\end{proof}

\begin{remark}  \label{ReUnbddComp}
If $A \subseteq B, \, A, B \in  \mathscr{A}(X) $ then each unbounded component 
of $X \setminus B$ is contained in an unbounded component of $ X \setminus A$.
\end{remark}

\begin{lemma} \label{LeUnbddComp}
Let $X$ be a connected, locally connected space.
Assume $ A \subseteq B, \ A, B \in \mathscr{A}^{*}(X)$. Then each unbounded component of $X \setminus B$ 
is contained in an unbounded component of $ X \setminus A$ and each unbounded component
of $ X \setminus A$ contains an unbounded component of $X \setminus B$.
\end{lemma}

\begin{proof}
Suppose first that $ A \subseteq K, \ K \in \mathscr{K}(X)$. 
The first assertion is Remark \ref{ReUnbddComp}. Now suppose to the contrary that $E$ is an unbounded component of 
$X \setminus A$ which contains no unbounded components of $X \setminus K$. Then $E$ is contained 
in the union of $K$ and all bounded components of $X \setminus K$. 
By Corollary  \ref{CoBddComp} this union is a bounded set, and so is $E$, which gives a contradiction. 

Now suppose $ A \subseteq U, \ U \in \mathscr{O}^{*}(X)$, so  $ K = \overline U$ is compact.
 Let $E$ be an unbounded component of $X \setminus A$. By the previous part, $ E$ contains an unbounded component $Y$ of  
$ X \setminus K$. But $Y \subseteq G$ for some unbounded component $G$ of $ X \setminus U$. Then $G \subseteq E$. 
\end{proof}

\begin{lemma} \label{LeNoUnbdComp}
Let $X$ be locally compact, connected, locally connected.
Let $A \in \mathscr{A}^{*}(X)$. Then the number of unbounded components of $ X \setminus A$ is finite.
\end{lemma}

\begin{proof}
Suppose first that $ A \in \mathscr{K}(X)$. By Lemma \ref{LeConLC}
let $U \subseteq \mathscr{O}_{c}^{*}(X)$ be such that $ A \subseteq U $. Then the assertion follows from
Lemma \ref{LeAaCompInU}.
Now suppose that $ A \in \mathscr{O}^{*}(X)$.  Then $\overline A \in \mathscr{K}(X)$, so the number of unbounded components
of $ X \setminus \overline A$ is finite. From Lemma \ref{LeUnbddComp} 
it follows that the number of unbounded components of $X \setminus A$  is also finite.
\end{proof}  

\begin{lemma} \label{LeCleverSet}
Let $X$ be locally compact, connected, locally connected.
Suppose $D \subseteq U$ where $ D \in \mathscr{K}(X), \ U \in \mathscr{O}^{*}(X).$
Let $C$ be the intersection of the union of bounded components of $X \setminus D$ 
with the union of bounded components of $X \setminus U$. Then  $C$ is compact 
and $ U \sqcup C$ is open.
\end{lemma}

\begin{proof}
Write 
$ X \setminus D = V \sqcup W,$ 
where $V$ is the union of bounded components of $ X \setminus D$, and $W$ is the union of unbounded components of  $ X \setminus D.$
Also write
$ X \setminus U = B \sqcup F,$ 
where $B$ is the union of bounded components of $ X \setminus U$, and $F$ is the union of unbounded components of $X \setminus U.$
By Lemma \ref{LeNoUnbdComp} $F$ is a closed set. Let 
$ C = V \cap B.$ 
Clearly, $C$ and $U$ are disjoint. 
Note that $U \sqcup B = X \setminus F$ is open, so
$ U \sqcup C = U \sqcup(V  \cap B) = (U \cup V) \cap (U \sqcup B)$
is also open. Now we shall show that  $X \setminus C$ is open.  Note that $ F \subseteq W$ by Remark \ref{ReUnbddComp}.  
The set $W$ is open by Lemma \ref{CmpCmplBdA}. Now 
$
X \setminus C = X \setminus (B \cap V) = (X \setminus B) \cup (X \setminus V)  
= ( U \sqcup F) \cup (D \sqcup W) = (U \cup D) \cup (F \cup W) = U \cup W
$
is an open set.  By Corollary \ref{CoBddComp} the set $C$ is bounded. 
\end{proof}

\begin{remark}
Lemma \ref{LeAaCompInU} is stated without proof in  \cite[ Lemma 3.4]{Aarnes:LC}. 
Lemma \ref{LeNoUnbdComp}  is close to  \cite[Lemma 3.5]{Aarnes:LC}, and Lemma \ref{LeCleverSet} is related to a part in the proof of Lemma 5.9 in \cite{Aarnes:LC}. 
\end{remark}



\section{Solid hulls}  \label{SolidSemisolid}

\begin{definition}
A set $A$ is semisolid if $A$ is connected, and $X \setminus A $ has only finitely many components. 
If $X$ is locally compact, noncompact, a set $A$ is solid if $A$ is  connected, and $X \setminus A$ has only unbounded components.
If $X$ is compact, a set $A$ is solid if $A$ and $X \setminus A$ are connected.
\end{definition}

\begin{example}
If $X= [0,1]^2$ is the unit square, the smaller square $A = [1/4, 3/4]^2$ is solid, and its boundary is not solid, but is semisolid. 
\end{example}

\begin{remark} \label{ReFinNoComp}
Let $X$ be noncompact locally compact, locally connected, connected.
From Lemma \ref{LeNoUnbdComp}
it follows that a bounded set $B$ is semisolid if and only if the number of bounded 
components of $X \setminus B$ is finite. For a  bounded solid set $A$, 
$ X \setminus A = \bigsqcup_{i=1}^n  E_i, $
where $ n \in \mathbb{N}$ and $E_i$'s are unbounded connected components.
\end{remark}

\begin{lemma} \label{SolidCompoLC}
Let $X$ be locally compact, locally connected, connected.
If $A \in \mathscr{A}_{c}^{*}(X)$ then each bounded component of $X \setminus A$ is a solid bounded set.
\end{lemma}

\begin{proof}
Let 
$ X \setminus A = \bigsqcup_{i \in I} B_i \sqcup \bigsqcup_{j \in J} D_j $
be the decomposition  of $X \setminus A$ into components, where $B_i$'s are bounded
components, $D_j$'s are unbounded ones (and  $J = \emptyset$ when $X$ is compact).
Pick a bounded component $B_k$. Then 
$ X \setminus B_k = A  \sqcup \bigsqcup_{i \neq k} B_i  \sqcup \bigsqcup_{j \in J} D_j. $ 
The set on the right hand side is connected by Lemma \ref{CmpCmplBdA};  it is also unbounded if $X$ is noncompact. Hence, $B_k$ is solid.
\end{proof}

A set $A \in \mathscr{A}_{c}^{*}(X)$ may not be solid. But we may make it solid by filling in the "holes" by adding to $A$ all bounded components 
of $X \setminus A$. More precisely, we have
  
\begin{definition} \label{solid hull}
Let $X$ be locally compact, locally connected,  connected.
For $A \in \mathscr{A}_{c}^*(X)$ let  $\{A_i\}_{i=1}^n$be the unbounded components 
of $X \setminus A$ and $\{B_t\}_{t \in T}$ be the bounded components of $X \setminus A$. 
We say that $\widetilde{A}=  A \sqcup \bigsqcup_{t \in T} B_t= X \setminus \bigsqcup_{i=1}^n A_i $ is a solid hull of $A$.
\end{definition}

\begin{remark}  \label{leSolidHu}
The set $\widetilde A $ is connected by Lemma \ref{CmpCmplBdA}.
If $X$ is noncompact, $X \setminus \widetilde A$ has only unbounded components, so $\widetilde A $ is solid. 
If $X$ is compact then  $\widetilde{A}= X$ for any connected closed or connected open set $A$. 
\end{remark}

The next lemma gives some properties of solid hulls. 

\begin{lemma}\label{PrSolidHuLC}
Let $X$ be noncompact locally compact, connected, locally connected.
Let $A, B \in \mathscr{A}_{c}^{*}(X)$.
\begin{enumerate}[label=(a\arabic*),ref=(a\arabic*)]
\item \label{part1}
If $ A \subseteq B$ then $\widetilde{A} \subseteq \widetilde{B}.$
\item \label{part2}
$\widetilde{A}$ is a bounded solid set, $A \subseteq  \widetilde{A}$, and $A$ is solid iff  $A   = \widetilde{A}.$
\item \label{part3}
$\widetilde{\widetilde{A}} = \widetilde{A}.$
\item  \label{part4}
If $A$ is open, then so is $ \widetilde{A}$. If $A$ is compact, then so is $\widetilde{A}.$
\item \label{part5}
If  $A, B$ are disjoint bounded connected sets, then their solid hulls $\widetilde{A}, \widetilde{B}$ are either disjoint or one is properly contained in the other.
\end{enumerate}
\end{lemma}

\begin{proof}
Part \ref{part1} follows since each unbounded component of 
$ X \setminus B$ is contained in an unbounded component of $  X \setminus A$. 
 If $A$ is compact, choose by Lemma \ref{LeConLC} a set $U \in \mathscr{O}_{c}^{*}(X)$ that contains $A$.
Since $\widetilde A$ is a union of $A$ and bounded components of $X \setminus A$, applying 
Lemma \ref{LeAaCompInU} we see that $\widetilde A$ is bounded. The rest of  parts \ref{part2} and \ref{part3} is immediate. 
For part \ref{part4}, note that if $A$ is open (closed) then 
each of finitely many (by Lemma \ref{LeNoUnbdComp}) unbounded components of $X \setminus A$ is closed (open) by Lemma \ref{CmpCmplBdA}.
To prove part \ref{part5}, let $A, B \in \mathscr{A}_{c}^{*}(X)$ be disjoint. 
If $A \subseteq \widetilde B$ then  $\widetilde A \subseteq \widetilde B$ by parts \ref{part1} and \ref{part3}. 
To prove that the inclusion is proper, assume to the contrary that $\widetilde A = \widetilde B$. If one of the sets $A, B$ 
is open and the other is closed, this equality means that $\widetilde A$ is a proper clopen subset of $X$, 
which contradicts the connectedness of $X$. Suppose $A$ and $B$ are both closed (both open).
Then it is easy to see that $A = E$, where $E$ is a  bounded component of $X \setminus B$,
an open (closed) set. Thus, $A$ is a proper clopen subset of $X$, which
contradicts the connectedness of $X$. Therefore, $\widetilde A$ is properly contained in $\widetilde B$. 
Similarly, if $ B \subseteq \widetilde A$ then $ \widetilde B \subseteq \widetilde A $, and the inclusion is proper. 
Suppose neither of the above discussed cases $A \subseteq \widetilde B$ or $B \subseteq \widetilde A$ occurs.  
Then by connectedness we have $ A \subseteq G$, $ B \subseteq E$,
where $G$ is an unbounded component of $ X \setminus B$ and $E$ is an unbounded 
component of $ X \setminus A$. Then $ B \subseteq \widetilde B \subseteq X \setminus G \subseteq X \setminus A$,
i.e. $\widetilde B$ is contained in a component of $ X \setminus A$.
Since $\widetilde B$ is connected and $ B \subseteq E$ we must have $ \widetilde B \subseteq E  \subseteq X \setminus \widetilde A$.
\end{proof}

\begin{remark} \label{ordering}
Suppose disjoint sets $A_1, A_2, \ldots, A_n  \in \mathscr{A}_{c}^{*}(X)$.
 On 
$\{ A_1, A_2, \ldots, A_n \}$ consider a partial order where $A_i \le A_j$ iff  $\widetilde{A_i} \subseteq \widetilde{A_j}$. 
(See Lemma \ref{PrSolidHuLC}.)
Let $ A_1, \ldots, A_p$ where $ p \le n$ be maximal elements in  $\{ A_1, A_2, \ldots A_n \}$
with respect to this partial order. 
Notice that $\widetilde{A_1}, \ldots, \widetilde{A_p}$ are all disjoint by part \ref{part5} of  Lemma \ref{PrSolidHuLC}.
For a maximal element $A_k,  k \in \{ 1, \ldots, p\}$  let
$$ I_k = \{ i \in \{ p+1, \ldots, n\}: \,  A_i \mbox{ is contained in a bounded component of }  X\setminus A_k \}. $$
The sets $I_k, \ k=1, \ldots, p$ are disjoint
(otherwise,  if $ i \in I_k \cap I_m, \  1\le k, m \le p$  then $\widetilde{A_k} \cap \widetilde{A_m} \ne \emptyset$).
Then $ \{1, \ldots, n\} = \{1, \ldots, p\} \sqcup \bigsqcup_{k=1}^p I_k $. 
Indeed, if $i \in \{1, \ldots, n\} \setminus \{ 1, \ldots, p\}$ we must have $ A_i \subseteq \widetilde A_i \subseteq \widetilde A_k$ 
for some maximal element $A_k$ (where $ k \in \{1, \ldots, p\}$), and  since $A_i$ and $A_k$ are 
disjoint, $A_i$ must be contained in a bounded component of  $A_k$, i.e. $ i \in I_k$.
Note that $A_1 \sc \ldots \sc A_n \se \widetilde{A_1}  \sc \ldots \sc \widetilde{A_p}$.  
\end{remark}

\begin{remark}
The closure of a solid set need not be solid. For example, in the infinite strip
$X = \mathbb{R} \times [0,4] $  the open set $ U = ((1,3) \times (0,4)) \cup ((5,7) \times(0,4)) \cup ((2,6) \times (1,3))$
is solid,  while its closure is not. However, we have the following result.
\end{remark}

\begin{lemma} \label{LeCsInside}
Let $X$ be locally compact, connected, locally connected.
\begin{enumerate}
\item
Suppose $X$  is noncompact. If  $K \subseteq U, \ K \in \mathscr{K}(X), \ U \in \mathscr{O}_{s}^{*}(X)$ 
then there exist  $ W \in \mathscr{O}_{s}^{*}(X)$ and $C \in \mathscr{K}_{s}(X)$ such that
$ K \subseteq W \subseteq C \subseteq U.$
\item \label{KWsol}
Suppose $X$  is noncompact.
If $K \subseteq V, \ K \in \mathscr{K}_{s}(X), \ V \in \mathscr{O}(X)$ then there exists $ W \in \mathscr{O}_{s}^{*}(X)$ such that 
$ K \subseteq W \subseteq \overline W \subseteq V.$
\item
Suppose $X$ is compact. If $K \subseteq U$,  $K \in \csx, \ U \in \ox$ or $ K \in \cx, \ U  \in \osx$ 
then there exists $ V  \in \osx$  and $C \in \csx$ such that
$ K \subseteq V \subseteq C \subseteq U.$ 
\end{enumerate}
\end{lemma}

\begin{proof}
1.)
One may take $W$ to be the solid hull of the set $V$  and $C$ to be the solid hull of the set $\overline V$, 
where $V$ is from Lemma \ref{LeConLC}. Then $K \se W \se C \subseteq U$  by Lemma \ref{PrSolidHuLC}.  \\
2.) 
By Lemma \ref{LeConLC} we may choose $ U  \in \mathscr{O}_{c}^{*}(X)$ such that
\begin{align} \label{V}
K \subseteq U \subseteq \overline U \subseteq V.
\end{align}
Since $ K \in \mathscr{K}_{s}(X)$ let
$ X \setminus K = \bigsqcup_{j=1}^n V_j$
be the decomposition into connected components.  Each $V_j $ is an unbounded 
open connected set. Since $X\setminus U \subseteq X \setminus K$, for $j=1, \ldots, n$ let 
$E_j$ be the union of all  bounded components of $X \setminus U$ contained in $V_j$, and 
let $F_j$ be the union of (finitely many by Lemma \ref{LeNoUnbdComp}) 
unbounded components of $X \setminus U$ contained in $V_j$.
By Lemma \ref{CmpCmplBdA} each $F_j$ is closed.
By Lemma \ref{LeUnbddComp} each $F_j$ is nonempty. 
Then by Lemma  \ref{CmpCmplBdA}  nonempty set $F_j \cap \overline U \subseteq V_j$.
Now, $E_j \subseteq \widetilde U$, so $ E_j $ is bounded.
Note that $X = K \sqcup \bigsqcup_{j=1}^n V_j$, and a limit point $x$
of $E_j$ can not be in $V_i$ for $i \neq j$; and it can not be in $K$, since in this case a neighborhood $U$ of $x$ contains no points of 
$E_j$. Thus, $\overline{ E_j} \subseteq V_j$. Then $(F_j \cap \overline U) \cup \overline{ E_j}$ is a compact set contained in $V_j$. 
By Lemma \ref{LeConLC}  there exists $ D_j \in \mathscr{K}_{c}(X)$ such that 
\begin{eqnarray} \label{3sh}
(F_j \cap \overline U) \cup \overline{ E_j} \subseteq D_j \subseteq V_j. 
\end{eqnarray} 
Let 
$ B_j = D_j \cup F_j.$
Then $B_j$ is connected because from (\ref{3sh}) one sees that $D_j $ intersects every component comprising $F_j$. 
Thus,  each $B_j$ is an unbounded closed connected set, $B_j \cap K =\emptyset$. Set
 $B= \bigcup_{j=1}^n B_j.$
Then $ X \sm U \se B$ and $B \cap K = \emptyset$. 
Since $ K \subseteq X \setminus B$, let $O$ be the connected component  of $X \setminus B$ such that 
$ K \subseteq O \subseteq X \setminus B$.
Since $B= \bigcup_{j=1}^n B_j \subseteq X \setminus O$, $B$ is contained in the union of unbounded components of $X \setminus O$. 
Hence, each bounded component of $ X \setminus O$ is disjoint from $B$, and so $\widetilde O \subseteq X \setminus B$. Thus
 $ K \subseteq O \subseteq \widetilde O \subseteq  X \setminus B \subseteq U.$
By (\ref{V}) we see that 
$ K \subseteq \widetilde O \subseteq U \subseteq \overline U \subseteq V,$
and we may take $W =\widetilde O$.  \\
3.) 
We will prove the statement for the first case  $K \in \csx, \ U \in \ox$,
and the second case will follow immediately by considering complements of sets.
$K$ is connected, so by taking a component of $U$ containing $K$ we  may assume that
$U \in \ocx.$ We have $X \setminus U \subseteq X \setminus K $, where 
$ X \setminus K  \in \osx$.
By Lemma \ref{LeConLC}  there exists $W  \in \ocx$ such that 
$X \setminus U  \subseteq W \subseteq \overline W \subseteq X \setminus K,$
so 
$ K \se X \sm \cl W \se X \sm W \se U.$
Since $ \ol W $ is connected, by 
Lemma \ref{CmpCmplBdA} the components of $X \sm \ol W$ are solid. 
Let $ V \in \osx$ be the component of  $X \sm \cl W$ that contains $K$.
Note that $ V \se  X \sm \cl W \se X \sm W \in \cx$, so $ \cl V \se X \sm W$. Then 
$ K \se V \se \cl V \se X \sm W \se U.$
Similarly we can get $ W_1 \in \osx$ such that 
$ K \se W_1 \se \cl W_1 \se V. $
Then  $ X \sm V \se X \sm \cl W_1 , \ X \sm V \in \csx, \ X \sm \ol {W_1} \in \ox$.
As above, we find $E  \in \osx$ such that  
$ X \sm V \se E \se X \sm \ol{W_1}.$
Then 
$K \se W_1 \se \ol{W_1} \se X \sm E  \se V \se \ol V  \se U.$
Since $W_1 \in \osx$ and $ X \sm E \in \csx$, this finishes the proof.
\end{proof}

In the spirit of Lemma \ref{LeCsInside} and Lemma \ref{LeConLC} we have
 
\begin{lemma} \label{ossreg}
Let $X$ be locally compact, connected, locally connected. Suppose $K \subseteq W,$ $ K \in \mathscr{K}(X),$ $W \in \mathscr{O}_{ss}(X)$. 
Then there exist $V \in \mathscr{O}_{ss}^{*}(X)$ and $ D \in \mathscr{K}_{ss}(X)$ such that 
$ K \subseteq V \subseteq D \subseteq W.$
\end{lemma}

\begin{proof}
Suppose first that $X$ is non-compact.
By Lemma \ref{LeConLC} choose $U \in \mathscr{O}_{c}^{*}(X)$ and $ C \in \mathscr{K}_{c}(X)$ such that 
$ K \subseteq U \subseteq C \subseteq W.$
Let $X \setminus W = \bigsqcup_{i=1}^n E_i, \ X \setminus C = \bigsqcup_{t \in T} V_t, \ X \setminus U = \bigsqcup_{s \in S} D_s$ 
be decompositions into connected components of $X \setminus W, \ X \setminus C,  \ X \setminus U$ respectively. Then 
$  \bigsqcup_{i=1}^n E_i \subseteq  \bigsqcup_{t \in T} V_t \subseteq \bigsqcup_{s \in S} D_s.$
Let  $T_0 = \{ t \in T: \ V_t \mbox{   is unbounded} \}$. 
Let us index by $T'$ the family of 
all bounded components of $X \setminus C$ each of which contains a component of $X \setminus W$. 
So $T'$ is  finite, and  $ \bigsqcup_{i=1}^n E_i \subseteq  \bigsqcup_{t \in T_0}  V_t \sqcup   \bigsqcup_{t \in T'} V_t$. 
Now let us index by $S'$ the family of all bounded components of $X \setminus U$ each of which contains a component $V_t$ for some $t \in T' $. 
Note that $S'$ is a finite index set and 
$\bigsqcup_{t \in T'} V_t \subseteq \bigsqcup_{s \in S'} D_s.$
Consider 
$ V = \widetilde U \setminus \bigsqcup_{s \in S'} D_s.$ 
Then $V$ is bounded. Also, $V$ is open. By Lemma \ref{CmpCmplBdA} $V$ is connected. Since 
$ X \setminus V = (X \setminus \widetilde U) \sqcup\bigsqcup_{s \in S'} D_s \subseteq \bigsqcup_{s \in S} D_s  = X \setminus U$
we see that $V \in \mathscr{O}_{ss}^{*}(X)$ (as the first equality indicates that $X \setminus V$ has finitely 
many components), and that $U \subseteq V$.
Now consider 
$D = \widetilde C \setminus \bigsqcup_{t \in T' } V_t.$
Then $D$ is compact. By Lemma \ref{CmpCmplBdA} $D$ is connected. We have 
$ X \setminus D 
= (X \setminus \widetilde C) \sqcup \bigsqcup_{t \in T'} V_t  \subseteq  (X \setminus \widetilde U) \sqcup \bigsqcup_{s \in S'} D_s = X \setminus V,$
so $X \setminus D$ has finitely many components, and $V \subseteq D$. Thus, $D \in \mathscr{K}_{ss}(X)$. Also, 
$ X \setminus W = \bigsqcup_{i=1}^n E_i \subseteq 
\bigsqcup_{t \in T_0} V_t \sqcup \bigsqcup_{t \in T'} V_t  = (X \setminus \widetilde C) \sqcup \bigsqcup_{t \in T'} V_t = X \setminus D.$
Therefore, $ D \subseteq W$. Then we have:
$ K \subseteq U \subseteq V \subseteq D \subseteq W,$
where $ V \in \mathscr{O}_{ss}^{*}(X)$ and $ D \in \mathscr{K}_{ss}(X)$.

If $X$ is compact use $\widetilde U = \widetilde C = X, \,T_0 = \emptyset$ and the same (but simplified) argument. 
\end{proof}

Let $V$ be an open subset of $X$ endowed with the subspace topology.  Let $D \subseteq V$.
By $\overline D^V$ we denote the closure of  $D$ in $V$ with the subspace topology. As before, 
$\overline D$ stands for the closure of $D$ in $X$.  

\begin{lemma} \label{76a}
Let $V \in \mathscr{O}(X), \ D \subseteq V$. Suppose $V$ is endowed with the subspace topology.
\begin{itemize}
\item[a)]
If $D$ is bounded in $V$ with the subspace topology then 
$\overline D^V = \overline D$ and $\overline D \subseteq V$. 
\item[b)]
If $D$ is bounded in $X$ and $\overline D \subseteq V$ then $D$ is bounded in $V$.
\end{itemize}
\end{lemma}

\begin{proof}
(a)
Clearly, $\overline D^V \subseteq \overline D$, $\overline{ \overline D^V} \se \overline D$, and $ D \subseteq \overline D^V$,  so we have
$ \overline D \subseteq  \overline{ \overline D^V} \subseteq \overline D.$
Thus, $ \overline D =  \overline D^V  \subseteq V$. 
(b)
Let $ \overline D \se V$. Again, $ \overline D =  \overline D^V $. Since $\overline D$ is compact in $X$, $\overline D^V$ is compact in $V$.
\end{proof}

\begin{remark} \label{bddInV}
Let  $V  \in \ox$ be endowed with the subspace topology.  \\
(1) From Lemma \ref{76a} we see that $D$ is bounded in $V$ iff $\overline D \subseteq V$. 
Hence, $D$ is unbounded in $V$ iff $\overline D \cap (X \setminus V) \neq \emptyset$.
If $X$ is compact and $V=X$ this criteria shows (as expected) that there are no unbounded components in $X$.  \\
(2) If $E$ is connected in $V$ endowed with the subspace topology then $E$ is connected in $X$. 
\end{remark}

The next two results give relations between being a solid set in a subspace of $X$ and being a solid set in $X$. 

\begin{lemma} \label{LeSolidInV}
Let $X$ be  locally compact, locally connected, connected.  Let $C \subseteq V, \ C \in \mathscr{C}_{s}(X), \ V \in \mathscr{O}(X)$. 
Suppose $V \sm C = \bsc_{t \in T} V_t$ is the decomposition into disjoint open connected sets.
Then $\cl V_t \cap (X \sm V) \neq \O$ for each $t$, and  $C \in \mathscr{C}_{s}(V)$.
\end{lemma}

\begin{proof}
If $X$ is compact and $V=X$ then $V \sm C$ has only one component $U = X \sm C$, and $\cl U \cap C \neq \O$ 
because otherwise $X$ is disconnected. 
We assume now that $X$ is noncompact or  $X$ is compact and $V \ne X$.
Write
$ X \setminus C = (X \setminus V)  \sqcup \bigsqcup_{t \in T} V_t.   $
Assume to the contrary that there exists $r \in T$ such that  $\overline{ V_r} \cap (X \setminus V) =\emptyset$.
By Remark \ref{OpenComp} $\overline{V_r} \cap V_t = \emptyset$ for each $t \ne r$.
Thus, $ \overline{ V_r} \subseteq C \sqcup V_r$.
Since $V_r \subseteq X \setminus C$ and $V_r$ is connected in $X$, assume  that $V_r$ is contained in a component $U$ of $X \setminus C$.
(If $X$ is compact then $U= X \sm C$).
Then $ V_r \subseteq U \cap \overline{V_r}  \subseteq U \cap (C \sqcup V_r) = V_r$, so $ U \cap \overline{V_r} = V_r$.
Thus, $U =(U \cap \overline{V_r}) \sqcup (U \setminus \overline V_r) =  V_r \sqcup (U \setminus \overline V_r)$ 
is the disconnection of $U$, unless $U = V_r$.
Then $U=V_r$ is a component of $X \setminus C$. If $X$ is noncompact this is impossible, since $V_r$ is bounded and $C$ is solid; 
if $X$ is compact this is impossible since $V_r = U =X \sm C$ implies $X \sm V= \O$.
Thus,  $\cl V_t \cap (X \sm V) \neq \O$ for each $t$. 
If we take $V \sm C = \bsc_{t \in T} V_t$ to be the decomposition into connected components in $V$ endowed with the subspace topology, then
by Remark \ref{bddInV} each $V_t$ is unbounded in $V$, i.e. $C \in \mathscr{C}_{s}(V)$. 
\end{proof}

\begin{lemma} \label{solidV2X}
Suppose $X$ is  locally compact, locally connected, connected.
Let $A \subseteq V, \ V \in \mathscr{O}_{s}^{*}(X)$. If $A \in \mathscr{A}_{s}(V) $ then $A \in \mathscr{A}_{s}^{*}(X)$.
\end{lemma}

\begin{proof}
First suppose $X$ is noncompact.
If $ A \in \mathscr{A}_{s}(V)$ then $A$ is connected in $X$ and bounded in $X$. Since $V \in \mathscr{O}_{s}^{*}(X)$, 
write $X \setminus V = \bigsqcup_{i \in I} F_i$ where $F_i$ are unbounded connected components. 
Let $V \setminus A = \bigsqcup_{t \in T} E_t$ be the decomposition into connected components in $V$.
Each $E_t$ is unbounded in $V$, i.e. $\overline{ E_t} \cap (X \setminus V) \ne \emptyset$, hence, 
$\overline{ E_t} \cap F_i \ne \emptyset$ for some $ i \in I$. 
Let $I' = \{ i \in I: \ F_i \cap \overline{ E_t} \ne \emptyset 
\mbox{  for some   } E_t \}$, and for $i \in I'$ let $T_i = \{ t \in T : \ \overline{ E_t} \cap F_i \ne \emptyset \}$. 
For $i \in I'$ the set $ F_i \cup \bigsqcup_{t \in T_i} E_t$ is unbounded and connected.
Since 
$ X \setminus A = (X \setminus V) \sqcup (V \setminus A) = 
\bigsqcup_{i \in I'} (F_i \cup \bigsqcup_{t \in T_i} E_t)  \sqcup \bigsqcup_{i \in I \setminus I'} F_i $
is a  disjoint union of unbounded connected sets, the statement follows.

For $X$ compact we use a simplified version of the same argument: write $X \setminus V = F$ where $F$ is a closed connected set, and note that 
$ X \setminus A =F \sqcup \bigsqcup_{t \in T} E_t$  is connected.
\end{proof}

\begin{example} 
Let $X = \{ z \in \mathbb{C}: 1 \le |z| \le 4 \}, \, V=  \{ z \in \mathbb{C}: 1 \le |z| < 3 \}, \, 
B = \{ z \in \mathbb{C}: 1 \le |z| \le 2 \}, \, W=  \{ z \in \mathbb{C}: 2<  |z| < 3 \}$. 
Then  $V \in \mathscr{O}_{s}^{*}(X), \ B \in \mathscr{C}_{s}(V) $ and $B \in \mathscr{K}_{s}(X) $, 
while $W \notin  \mathscr{A}_{s}(V)$ and is not solid in $X$. 
(Note that $W   \in \mathscr{O}_{ss}^{*}(X)$, which will follow from Lemma \ref{DecUssN} below.) 
\end{example}

\section{Structure of solid and semisolid sets} \label{Sstructure}

Now we shall take a closer look at the structure of open solid or semisolid sets that contain
closed solid or closed connected sets.

\begin{lemma} \label{LeDecompV}
Let $X$ be locally compact, connected, locally connected.
Let $ C \subseteq V,  \ C \in \mathscr{K}_{s}(X)$. 
\begin{enumerate}[label=(\roman*),ref=(\roman*)]
\item
Suppose  $V \in \mathscr{O}_{s}^{*}(X)$. If $V \setminus C$ is connected then 
$V = C \sqcup W$   where   $ W \in \mathscr{O}_{ss}^{*}(X).$
If  $ V \setminus C$ is disconnected then 
$ V = C \sqcup \bigsqcup_{i=1}^n V_i$     where   $ V_i \in \mathscr{O}_{s}^{*}(X), \ i=1, \ldots, n, \ n \in \N. $
\item
Suppose $ V \in \mathscr{O}_{ss}^{*}(X)$.  Then
$ V = C \sqcup \bigsqcup_{i=1}^n V_i$    where     $ V_i \in \mathscr{O}_{ss}^{*}(X), \ i=1, \ldots, n, \ n \in \N.  $
\end{enumerate}
\end{lemma}
 
\begin{proof} 
(i)
Suppose  $V \in \mathscr{O}_{s}^{*}(X)$ and let 
$X \setminus V = \bigsqcup_{j \in J} F_j$
be the decomposition 
into connected components, so $J$ is a finite index set and  each $F_j$ is unbounded 
(respectively, $X \setminus V = F$ where $F$ is connected if $X$ is compact).
If $V \setminus C$ is connected then taking $W = V \setminus C$ we see that
$X \setminus W = (X \setminus V) \sqcup C =C \sqcup \bigsqcup_{ j \in J} F_j $
has finitely many components, i.e. $W \in \mathscr{O}_{ss}^{*}(X)$. \\
Now assume that $V \setminus C$ is not connected.  
By Lemma \ref{LeSolidInV}  $C \in \mathscr{C}_{s}(V)$. 
The set $ V \setminus C$ is also disconnected in $V$, so using Remark \ref{ReFinNoComp} let
$V \setminus C = \bigsqcup_{i=1}^n V_i, \ n \ge 2 $ 
be the decomposition into connected (unbounded in $V$)
components in $V$. Each $V_i$ is connected in $X$.
To show that each $V_i \in \mathscr{O}_{s}^{*}(X)$ 
we only need to check that the components of $X \setminus V_i$ are unbounded
(respectively, that $X \setminus V_i$ is connected if $X$ is compact).
For notational simplicity, we shall show it for $V_1$. For $2 \le i \le n$ 
by  Lemma \ref{LeSolidInV} 
$\overline{ V_i}$ intersects $X\setminus V$, hence, intersects some $F_j$. 
Let  $J_1 = \{ j \in J:  F_j \cap \overline{ V_i} \neq \emptyset \mbox{  for some  }  2 \le i \le n \} $.
By Lemma \ref{CmpCmplBdA} the set $ C \sqcup \bigsqcup_{i=2}^n V_i $ is connected in $V$, so by Remark \ref{bddInV} and Remark \ref{OpenComp} 
the set $\bigsqcup_{j \in J_1} F_j \sqcup C \sqcup \bigsqcup_{i=2}^n V_i $ is connected. 
It is also unbounded.
Now 
\begin{eqnarray*}
X \setminus V_1 =  (X \setminus V) \sqcup (V \setminus V_1)  = \bigsqcup_{j \in J} F_j \sqcup C \sqcup \bigsqcup_{i=2}^n V_i 
&=& (\bigsqcup_{j \in J_1} F_j \sqcup C \sqcup \bigsqcup_{i=2}^n V_i ) \sqcup \bigsqcup_{j \in J \setminus J_1} F_j.
\end{eqnarray*} 
Since $X \setminus V_1$ is the disjoint union of connected unbounded sets, it follows that
$V_1$ is solid. 
(If $X$ is compact $X \setminus V_1 =F \sqcup C \sqcup \bigsqcup_{i=2}^n V_i$ is connected. ) \\
(ii) 
Suppose $V \in \mathscr{O}_{ss}^{*}(X)$ and let $F_1, \ldots, F_k$ be the 
components of $X \setminus V$. 
Since $C$ is compact, by Lemma \ref{LeSolidInV} 
$C \in \mathscr{K}_{s}(V)$.
Let
$V \setminus C = \bigsqcup_{i=1}^n V_i, \ n \ge 1 $
be the decomposition into connected components in $V$ according to 
Remark \ref{ReFinNoComp}.  
Each $V_i$ is connected in $X$, and 
to show that each $V_i  \in \mathscr{O}_{ss}^{*}(X)$ we only need to check that $X \setminus V_i$ has finitely
many components. For simplicity, we shall show it for $V_1$. We have
$ X \setminus V_1 = (X \setminus V) \sqcup (V \setminus V_1)  = \bigsqcup_{j=1}^k F_j \sqcup C \sqcup \bigsqcup_{i\ne 1} V_i.$
Since $X \setminus V_1$ is a finite disjoint union of connected sets, the number of components
of $ X \setminus V_1$ is finite, so $V_1 \in \mathscr{O}_{ss}^{*}(X)$.
\end{proof} 

\begin{lemma} \label{finiteT}
Let $X$ be locally compact, connected, locally connected.
Suppose that 
$ V = \bigsqcup_{j=1}^m C_j \sqcup \bigsqcup_{t \in T} U_t, $
where $V \in \mathscr{O}_{ss}^{*}(X), \ C_i  \in \mathscr{K}_{s}(X), \ U_t \in \mathscr{O}_{c}^{*}(X)$. Then
 $T$ is finite, and each $U_t \in \mathscr{O}_{ss}^{*}(X)$.
\end{lemma}

\begin{proof}
The proof works for $X$ compact and noncompact and uses induction on $m$.  
Let $m=1$. Using Lemma \ref{LeDecompV} we have
$ V \setminus C_1 = \bigsqcup_{i=1}^n V_i = \bigsqcup_{t \in T} U_t.$
Since sets $V_i$'s and $U_t$'s are connected, $T$ must be finite.  Now let 
$V = \bigsqcup_{j=1}^m C_j \sqcup \bigsqcup_{t \in T} U_t $ and assume that the result holds for any 
bounded open semisolid set which contains less than $m$ compact solid sets. 
Using Lemma \ref{LeDecompV} we see that
$ V =  C_1 \sqcup \bigsqcup_{i=1}^n V_i =C_1 \sqcup \bigsqcup_{j=2}^m C_j \sqcup \bigsqcup_{t \in T} U_t,$
where $V_i \in \mathscr{O}_{ss}^{*}(X)$.
All involved sets are connected, so each set $V_i$ is a disjoint union of sets from the 
collection $\{ C_2, \ldots, C_m, U_t,  t \in T \}$. By the induction hypothesis each $V_i$ 
contains finitely many sets, and it follows that $T$ is finite.  Since $T$ is finite, we see that  for each $ t \in T$ the set
$X \setminus U_t =  (X \setminus V)  \sqcup \bigsqcup_{i=1}^m C_i \sqcup \bigsqcup_{r \in T,\, r \neq t} U_r$
has finitely many components, so $U_t \in \mathscr{O}_{ss}^{*}(X)$.
\end{proof}

\begin{lemma} \label{finiteSP}
Let $X$ be locally compact, connected, locally connected.
If $A = \bigsqcup_{t \in T} A_t$, $A , A_t \in \mathscr{A}_{s}^{*}(X)$  with at most finitely many
$A_t \in \mathscr{K}_{s}(X)$ then $T$ is finite. 
\end{lemma}

\begin{proof}
Assume first that $A \in \mathscr{O}_{s}^{*}(X)$. If the cardinality $|T| > 1$ then there must be  a compact solid set 
among $A_t$, and the result follows from Lemma \ref{finiteT}. Assume now that $A \in \mathscr{K}_{s}(X)$
and  write  
$ A = \bigsqcup_{j=1}^m C_j \sqcup \bigsqcup_{t \in T} U_t, $
where $C_j \in \mathscr{K}_{s}(X), \ U_t \in \mathscr{O}_{s}^{*}(X)$. 
By Lemma \ref{LeCsInside} choose $ V \in \mathscr{O}_{s}^{*}(X)$ such that $A \subseteq V$. 
By Lemma \ref{LeDecompV} we may write $V \setminus A =  \bigsqcup_{i=1}^n V_i$, where 
$ V_i \in \mathscr{O}_{ss}^{*}(X)$. Then 
$ V =\bigsqcup_{j=1}^m C_j \sqcup \bigsqcup_{t \in T} U_t  \sqcup  \bigsqcup_{i=1}^n V_i, $
and by Lemma \ref{finiteT} $T$ is finite.
\end{proof}

\begin{lemma} \label{LeDecompU}
Let $X$ be non-compact, locally compact, connected, locally connected. 
Suppose $ C \subseteq U, \ \  C \in \mathscr{K}_{c}(X), \ \ U \in \mathscr{O}_{s}^{*}(X)$. 
If $\  U \setminus \widetilde C$ is disconnected then
$ U = C \sqcup \bigsqcup_{s \in S} V_s$ where $V_s \in \mathscr{O}_{s}^{*}(X).$
If $ \  U \setminus \widetilde C$ is connected then
$ U = C \sqcup \bigsqcup_{s \in S} V_s \sqcup W$ where  $ V_s \in \mathscr{O}_{s}^{*}(X), \ W \in \mathscr{O}_{ss}^{*}(X).$
\end{lemma}

\begin{proof}
Note that $\widetilde C \in \mathscr{K}_{s}(X)$ and $\widetilde C \subseteq U$ by Lemma \ref{PrSolidHuLC}.  
Assume that $U \setminus \widetilde C$ is disconnected.  By Lemma \ref{LeDecompV} we may write
$ U = \widetilde C \sqcup \bigsqcup_{i=1}^n U_i$,  $ U_i \in \mathscr{O}_{s}^{*}(X)$.  But
$\widetilde C = C \sqcup \bigsqcup_{\alpha} V_{\alpha}$, where $V_{\alpha}$ are bounded components 
of $ X \setminus C$, so  by Lemma  \ref{SolidCompoLC} each $V_{\alpha} \in \mathscr{O}_{s}^{*}(X)$. 
After reindexing, one may write $ U = C \sqcup \bigsqcup_{s \in S} V_s$,  $V_s \in \mathscr{O}_{s}^{*}(X).$ 
The proof for the case  when 
$U \setminus \widetilde C$ is connected follows  similarly from  Lemma \ref{LeDecompV}. 
\end{proof}

\begin{lemma} \label{DecUssN}
Let $X$ be locally compact, connected, locally connected.
Suppose $ \bsc_{j=1}^n C_j  \subseteq V,  \ C_1, \ldots, C_n \in \mathscr{K}_{c}(X), \ V \in  \mathscr{O}_{ss}^{*}(X)$. 
Then $ V  = \bsc_{j=1}^n C_j  \sqcup \bigsqcup_{t \in T}  U_t $
where each $U_t \in  \mathscr{O}_{ss}^{*}(X)$ and all but finitely many are solid.
\end{lemma}

\begin{proof}
The proof is by induction on $n$.  Suppose  $n=1$, i.e. $ C \se V, C \in \mathscr{K}_{c}(X), \ V \in  \mathscr{O}_{ss}^{*}(X)$.  
Then $C \in  \mathscr{K}_{c} (V)$, where  $V$ is equipped with the subspace topology.
Being open, $V$ is a locally compact, locally connected subspace.
Let 
$V \sm C = \bigsqcup_{t \in T}  U_t $
be the decomposition into components in $V$, i.e. each $U_t \in  \mathscr{O}_{c}(V)$, and
so  $U_t \in  \mathscr{O}_{c}^{*}(X)$. 
By Lemma \ref{CmpCmplBdA} $C \cap \overline{U_t}^V \ne \O$, so also $  C \cap \overline{U_t} \ne \O$ for each $t$. 
Then
$ X \sm U_t = (X \sm V) \sc (C \sqcup \bigsqcup_{r \in T, r \ne t}  U_r),$
where the set in the last parenthesis is connected by Remark \ref{OpenComp}.
Since $V$ is semisolid, $ X \sm U_t $ has finitely many components, so $U_t \in  \mathscr{O}_{ss}^{*}(X)$ for each $t$.
By Lemma \ref{LeNoUnbdComp} all but finitely many of $U_t$'s are bounded in $V$, 
hence, by Lemma \ref{SolidCompoLC} belong to  $\mathscr{O}_{s}^{*}(V)$, hence, by Lemma \ref{solidV2X} belong to  $\mathscr{O}_{s}^{*}(X)$.
 
Assume the result for less than $n$ compact connected sets. For $n$ compact connected sets we have
$ V \sm  (\bsc_{j=1}^n C_j) = (V \sm C_1) \sm (C_2 \sc \ldots \sc C_n) = (\bsc_{t \in T} U_t)  \sm (C_2 \sc \ldots \sc C_n),$
where by the first part finitely many of the sets $U_t$ are in $ \mathscr{O}_{ss}^{*}(X)$ and  the rest are in  $ \mathscr{O}_{s}^{*}(X)$.
$C_2, \ldots, C_n$ are contained in finitely many of $U_t$'s, each of which contains less than $n$ disjoint compact connected sets, 
and the result follows by induction. 
\end{proof}

\begin{lemma}  \label{ossdec}
Let $X$ be locally compact, noncompact, connected, locally connected.
Suppose that $  \bigsqcup_{i=1}^n F_i  \se W$, where $F_1, \ldots, F_n \in \mathscr{K}_{s}(X), W \in  \mathscr{O}_{s}^{*}(X)$, and 
$ W \sm (F_1 \sc \ldots \sc F_n) $ is disconnected. Then there is a solid decomposition
$W = \bsc_{i \in I} F_i \sc \bsc_{j=1}^p U_j$, where $ \O \ne I \se \{ 1, \ldots, n\},  \, p \in \N, \  U_j \in  \mathscr{O}_{s}^{*}(X)$.  
\end{lemma}

\begin{proof}
By Lemma \ref{finiteT} write
$W = \bigsqcup_{i=1}^n F_i  \sc \bsc_{j=1}^m W_j, $ 
where $m \ge 2$, and $W_j  \in \mathscr{O}_{ss}^{*}(X)$.
Let $W_1, \ldots, W_p$ ($p \le m$) be maximal elements among $W_1, \ldots, W_n$ with respect to the partial order given by 
$W_j \le W_t$ iff $ \widetilde W_j \subsetneqq  \widetilde W_t$. 
By Lemma \ref{PrSolidHuLC} $\widetilde W_i \cap \widetilde W_t = \emptyset$ for any two maximal elements $W_j, W_t$.
Let $I$ be the index set for  $F_i$'s each of which is not contained in $ \bsc_{j=1}^p \widetilde W_j$. Sets $\widetilde W_j \se W$, so
$  \bsc_{j=1}^p \widetilde{W_j} \sc \bsc_{i \in I} F_i \se W.$
For a maximal element $W_j$ and a bounded component  $B$ of $X \sm W_j$, since $B \se W$ we have 
$ B \se  \bigsqcup_{i=1}^n F_i  \cup \bigsqcup_{k \ne j} W_k \se X \sm W_j.$ 
Each set from the family  $ \{ F_1, \ldots, F_n, W_k, \, k \ne j \}$ intersects only one component of $ X \sm W_j$, 
so $B$ is a union of some sets from this family. Then 
 $ \widetilde{W_j}$ is a union of some sets from the family $\{ F_1, \ldots, F_n, W_1, \ldots, W_m \}$. 
Together with Remark \ref{ordering} it gives
$  W =  \bsc_{j=1}^p \widetilde{W_j} \sc \bsc_{i \in I} F_i.$
If $I \ne \O$, the statement of the lemma follows with $U_j= \widetilde W_j$. 

Assume that $I = \O$. By connectedness, $W$ is the single maximal element among $W_1, \ldots, W_m$. Without loss of generality,  
$W = \widetilde{W_m}$, and so $ \widetilde W_i \subsetneqq \widetilde{W_m}$ for $ i =1, \ldots, m-1$.  
Among sets $W_1, \ldots, W_{m-1}$ again find maximal elements $W_1, \ldots, W_q$, $ q \le m-1$. 
As above, $W =  \widetilde{W_m} = \bsc_{j=1}^q  \widetilde{W_j} \sc \bsc_{i \in I'} F_i$.
If there is only one maximal element (say, $W_1$) then $I' \neq \O$ because  $ \widetilde{W_1} \subsetneqq \widetilde{W_m} $.
If there are at least two maximal elements then $I' \neq \O$ because of connectedness of $W$. 
\end{proof}

\begin{remark}
Part \ref{part5} of Lemma \ref{PrSolidHuLC}, part \ref{KWsol} of Lemma \ref{LeCsInside},  
and the last statement in Lemma \ref{LeSolidInV}    
are close to  \cite[Lemmas 3.8, 3.9, 4.2]{Aarnes:LC}. The existence of $V$ in the last part of Lemma \ref{LeCsInside} was first proved in 
\cite[Lemma 3.3]{Aarnes:ConstructionPaper}.
The case " $V \setminus C$ is disconnected" in the first part of Lemma \ref{LeDecompV}  is \cite[Lemma 4.3]{Aarnes:LC}, and 
Lemma \ref{finiteSP} is an expanded (to compact sets as well) version of  \cite[Lemma 4.4]{Aarnes:LC}. 
Our proofs are modified, expanded, or different.
\end{remark}

\begin{remark}
For a compact space, solid sets are connected and co-connected; they are essential in the construction of a topological measure 
from a solid-set function.
It is tempting to employ sets that are connected and co-connected in the locally compact setting, but  
this is not the right collection of sets for various reasons. 
For one, one may end up with topological measure $ \mu = 0$ starting from a non-trivial set function on sets 
that are connected and co-connected (see \cite[Example 6.2]{Aarnes:LC}). 
The key role in the locally compact noncompact case will be played by semisolid sets. 
For a compact space the definition of a semisolid set as a set whose complement has finitely many components was present in \cite{Grubb:Lectures}, 
but it is not clear, who coined the term. 
The same definition of semisolid set in locally compact spaces works well for our goals. 
The definition of solid sets is different  for compact spaces and locally compact noncompact spaces (\cite{Aarnes:ConstructionPaper}, \cite{Aarnes:LC}),
but as we shall see, it is what allows us to prove many results simultaneously for compact and noncompact spaces. 
Also, with the current definition of a solid set in the locally compact case, a bounded solid set in a locally compact space is solid 
in its one-point compactification (Lemma \ref{hatXsoli} below). 
\end{remark}

\section{Definition and basic properties of topological measures} \label{TM}

\begin{definition} \label{TMLC}
A topological measure on a locally compact space $X$ is a set function
$\mu:  \mathscr{C}(X) \cup \mathscr{O}(X) \to [0,\infty]$ satisfying the following conditions:
\begin{enumerate}[label=(TM\arabic*),ref=(TM\arabic*)]
\item \label{TM1} 
if $A,B, A \sqcup B \in \mathscr{K}(X) \cup \mathscr{O}(X) $ then
$
\mu(A\sqcup B)=\mu(A)+\mu(B);
$
\item \label{TM2}  
$
\mu(U)=\sup\{\mu(K):K \in \mathscr{K}(X), \  K \subseteq U\}
$ for $U\in\mathscr{O}(X)$;
\item \label{TM3}
$
\mu(F)=\inf\{\mu(U):U \in \mathscr{O}(X), \ F \subseteq U\}
$ for  $F \in \mathscr{C}(X)$.
\end{enumerate}
\end{definition} 

\begin{remark} 
It is important that in Definition \ref{TMLC} condition \ref{TM1} holds for sets from 
$\mathscr{K}(X) \cup \mathscr{O}(X)$. In fact, \ref{TM1} fails on $\mathscr{C}(X) \cup \mathscr{O}(X)$. See Example \ref{puncdisk} or 
Example \ref{linetm} below. 
\end{remark}

We have the following immediate properties of topological measures on locally compact spaces. 

\begin{lemma} \label{propTMLC}
The following is true for a topological measure:
\begin{enumerate}[label=(t\arabic*),ref=(t\arabic*)]  
\item \label{l1}
$\mu$ is monotone, i.e. if $ A \subseteq B, \ A, B \in  \mathscr{C}(X) \cup \mathscr{O}(X)$ then $\mu(A) \le \mu(B)$. 
\item \label{smooth}
If an increasing net $U_t \nearrow U$, where $U_t, U \in \mathscr{O}(X)$ then
 $\mu(U_t) \nearrow \mu(U)$.
In particular, $\mu$ is additive on $\mathscr{O}(X)$. 
\item \label{l5}
$\mu( \emptyset) = 0$.
\item \label{kl}
If $V \sqcup K \subseteq U$, where $U , V \in \mathscr{O}(X), \ K \in \mathscr{K}(X)$ then $\mu(V) + \mu(K) \le \mu(U).$
\item \label{l2}
If $\mu$ is compact-finite then $\mu(A) < \infty$ for each $A \in \mathscr{A}^{*}(X)$.
$\mu$ is finite iff $\mu$ is real-valued.
\item \label{CoRegulLC}
If $X$ is locally compact, locally connected then for any $U \in \mathscr{O}(X)$ 
$$\mu(U)=\sup \{ \mu(C): \ C \in \mathscr{K}_{0}(X), \ C\subseteq U\}.$$
\item \label{l8}
If $X$ is locally compact, connected, locally connected then  
$ \mu(X)  = \sup\{ \mu(K) :  \ K \in \mathscr{K}_{c}(X) \} = \sup\{ \mu(K) :  \ K \in \mathscr{K}_{s}(X) \} .$
\end{enumerate}
\end{lemma}

\begin{proof}
(t1)
Immediate from Definition \ref{TMLC}. 
(t2)
Suppose $U_t \nearrow U, U_t, U \in \mathscr{O}(X)$.  
Let compact $K \subseteq U.$   By Remark \ref{netsSETS}, there is $t' $ such that $K \subseteq U_s$ for all $t \ge t'$.
Then $\mu(K) \le \mu(U_t) \le \mu(U)$ for all $t \ge t'$, and we see from the inner regularity condition \ref{TM2} of Definition \ref{TMLC}
(whether $ \mu(U) < \infty$ or $ \mu(U) =\infty$)  that $\mu(U_t) \nearrow \mu(U)$. 
(t3)
Easy to see since $ \mu$ is not identically $ \infty$.  
(t4) 
Easy to see from part \ref{TM2} of Definition \ref{TMLC}. 
(t5)
If $U$ is an open bounded set then $ \mu(U) \le \mu(\overline U) < \infty$. 
(t6)
By Lemma \ref{LeCCoU} for arbitrary $K \subseteq U, \ K \in \mathscr{K}(X), \ U  \in \mathscr{O}(X)$ there is 
$C \in \mathscr{K}_{0}(X)$ with $ K \subseteq C \subseteq U$. By monotonicity 
$ \mu(U) = \sup\{ \mu(K): \ K \in \mathscr{K}, \ K \subseteq U\} 
\le  \sup\{ \mu(C): \ C \in \mathscr{K}_{0}(X), \ K \subseteq C \subseteq U\}  \le \mu(U). $
(t7)
Follows from Lemma  \ref{LeConLC} and Lemma \ref{PrSolidHuLC}.
\end{proof}

\begin{proposition} \label{PrFinAddLC}
Let $X$ be locally compact.
Consider the following conditions:
\begin{enumerate}[label=(c\arabic*),ref=(c\arabic*)]
\item \label{usl2}
$\mu(U) = \mu(K) + \mu(U \setminus K) $ whenever $K \subseteq U, \ K \in \mathscr{K}(X), \ U \in  \mathscr{O}(X).$
\item \label{usl1'}
$\mu(K \sqcup C ) = \mu(K) + \mu(C) $ for any disjoint compact sets $K,C.$
\item \label{usl1}
$\mu(U \sqcup V ) = \mu(U) + \mu(V) $ for any disjoint open sets $U,V.$
\end{enumerate}
For a set function $\mu: \mathscr{O}(X) \cup \mathscr{C}(X) \rightarrow [0,\infty] $ we have:
\begin{enumerate}
\item
If $ \mu$ satisfies \ref{TM2}  of Definition \ref{TMLC} then \ref{usl1'} $ \Rightarrow$ \ref{usl1}.
\item \label{compclo}
If $ \mu$ is monotone on open sets and  satisfies \ref{TM3}   then \ref{usl1} $ \Rightarrow$ \ref{usl1'}.
\item \label{coopA}
If $ \mu$ satisfies \ref{TM2}  and \ref{TM3}  then  \ref{usl1} and  \ref{usl1'} are equivalent.
\item
If $\mu$ satisfies \ref{TM2}, \ref{TM3},  \ref{usl2} and one of \ref{usl1'}, \ref{usl1} then \ref{TM1}  holds for $\mu$; hence, 
$\mu$ is a topological measure.
\end{enumerate}
\end{proposition}

\begin{proof}
Note that  \ref{TM2} implies monotonicity on $ \ox$ and  \ref{TM3} implies monotonicity on $\cx$. \\
(1.) 
For a compact $D \se U \sc V$ write $D = K \sc C$, where $ K \se U, C \se V$. Then by  \ref{usl1'} and  \ref{TM2} 
 $\mu(D)  = \mu(K) + \mu(C) \le \mu(U) + \mu(V)$, 
 so taking supremum over all $D \se U \sc V$ we have $\mu(U \sc V) \le \mu(U) + \mu(V)$. If $K \se U, C \se V$ then
 $\mu(K)  + \mu(C) = \mu(K \sc C) \le \mu(U \sc V)$, so by  \ref{TM2} $ \mu(U) + \mu(V) \le \mu(U \sc V)$. \\
(2.) Suppose $K, C$ are disjoint compact sets, $K \sc C \se W, \ W \in \ox$. $X$ is completely regular, so
choose disjoint open sets $U, V$ such that $ K \se U, C \se V, \, U \sc V \se W$. 
Then using \ref{TM3}, \ref{usl1} and monotonicity on open sets we have:
$ \mu(K) + \mu(C) \le \mu(U) + \mu(V) = \mu(U \sc V) \le \mu(W).$
Taking infimum over all $W$ containing $K \sc C$ we see that $ \mu(K) + \mu(C) \le \mu(K \sc C)$.
Since $\mu(K \sc C) \le \mu(U \sc V) = \mu(U) + \mu(V) $, by  \ref{TM3} we have $ \mu(K \sc C) \le \mu(K) + \mu(C)$. \\
(3.) Follows from previous parts. \\
(4.)  Our proof basically follows the proof of  \cite[Proposition 2.2]{Alf:ReprTh} where the result first appeared for 
compact-finite topological measures. 
By part \ref{coopA} we only need to check \ref{TM1} in the situation when $ A \in \mathscr{K}(X), \ B \in \mathscr{O}(X)$, and $ A \sqcup B$ is either 
compact or open. If  $ A \sqcup B$ is open then using condition \ref{usl2} we get
$ \mu(A \sqcup B) = \mu((A \sqcup B) \setminus A) + \mu(A) = \mu(B) + \mu(A).$
Now suppose $A \sqcup B \in \mathscr{K}(X)$. 
Let $ C \in \mathscr{K}(X), \ C \subseteq B$. Then finite additivity and monotonicity of $\mu$ on $\mathscr{K}(X)$ gives
$ \mu(A)  + \mu(C) = \mu(A \sqcup C) \le  \mu(A \sqcup B).$
By \ref{TM2}, 
$ \mu(A) + \mu(B) \le \mu(A \sqcup B).$
The opposite inequality is obvious if $\mu(A) = \infty$, so let $\mu(A) < \infty$, and  
for $ \epsilon >0$ pick $ U \in \mathscr{O}(X)$ such that 
$A \subseteq U$ and $\mu(U) < \mu(A) + \epsilon$. Then compact set $A \sqcup B$ is contained in 
the open set $B \cup U$. Also, the compact set $ (A \sqcup B) \setminus U = B \setminus U$ is contained in 
$ B \cup U$, and $ (B \cup U) \setminus (B \setminus U)  = U$. Applying \ref{TM2} and then 
\ref{usl2} we have:
$
\mu(A \sqcup B) \le \mu(B \cup U) = \mu((B \cup U) \setminus (B\setminus U)) + \mu(B \setminus U) 
= \mu(U) + \mu(B \setminus U)  \le \mu(U) + \mu(B)  
\le \mu(A)  + \mu(B) + \epsilon.
$
Thus, 
$ \mu(A \sqcup B)  \le \mu(A) + \mu(B) .$ This finishes the proof.
\end{proof}  

\begin{remark} \label{dualwrong}
Of course, any topological measure satisfies  \ref{usl2} of  Proposition \ref{PrFinAddLC}. It is interesting to note that a 
similar condition regarding a bounded open subset of a closed set fails for topological measures, i.e. 
$ \mu(F) = \mu(U) + \mu(F \setminus U) $
where $F$ is closed and $U$ is open bounded, in general is not true, as Example \ref{linetm} below shows. 
\end{remark}

Using complements of sets, we see that when $X$ is compact, the definition of a real-valued topological measure is equivalent to:

\begin{definition}\label{tmC}
A real-valued topological measure on a compact space $X$ is a set function
$\mu: \ox \cup \cx \rightarrow [0, \infty) $ that 
satisfies the following conditions:
\begin{enumerate}[label=(T\arabic*),ref=(T\arabic*)]
\item \label{T1} 
$ \mu(A\sc B)=\mu(A)+\mu(B) \, \mbox{       if       }\,  A, B, A \sc B \in \ox \cup \cx. $ 
\item[(T2)]
$\mu(U)=\sup\{\mu(C):C \in \cx, \ C\se U\} \, \mbox{       for       } \, U \in \ox. $
\end{enumerate} 
\end{definition}

\begin{remark}
By Lemma \ref{propTMLC}, Definition \ref{tmC} is equivalent to the definition of a (real-valued) topological measure given 
in all previous papers that use topological 
measures on compact spaces.
\end{remark}

When $X$ is compact we have a simpler version of Proposition \ref{PrFinAddLC} for real-valued $\mu$:

\begin{proposition} \label{PrFinAddC}
Let $X$ be compact. Consider the following condition:
\begin{enumerate}[label=(k\arabic*),ref=(k\arabic*)]
\item \label{usl2k}
$\mu(X) = \mu(K) + \mu(X \setminus K) $  for any $K \in \mathscr{C}(X).$
\item \label{usl1k'}
$\mu(K \sqcup C ) = \mu(K) + \mu(C) $ for any disjoint closed sets $K,C.$
\item \label{usl1k}
$\mu(U \sqcup V ) = \mu(U) + \mu(V) $ for any disjoint open sets $U,V.s$
\end{enumerate}
Suppose a set function $\mu: \mathscr{O}(X) \cup \mathscr{C}(X) \rightarrow [0,\infty)$ satisfies \ref{TM2} and \ref{TM3} of 
Definition \ref{TMLC}, \ref{usl2k} and one of \ref{usl1k'}, \ref{usl1k}.  Then  \ref{TM1} holds for $\mu$, so $\mu$ is a topological measure.
\end{proposition}

\begin{proof}
As in the proof of Proposition \ref{PrFinAddLC}, \ref{usl1k'}  and  \ref{usl1k} are equivalent.  
Then $\mu$ is finitely additive on open sets and on closed sets, and 
we only need to check \ref{TM1}  in the case when 
$C$ is closed, $U$ is open, and $C \sqcup U = D$ is closed or open.  
Condition \ref{usl2k} means that $\mu(X) = \mu(A) + \mu(X \setminus A)$ for any closed or open set $A$.
Suppose  $C \sqcup U = D$ is closed.
Then $X \setminus D$ is open, disjoint from $U$, so by \ref{usl1k} and \ref{usl2k} we have
$ \mu(U \sqcup (X \setminus D) ) = \mu(U) + \mu(X \setminus D) = \mu(U) + \mu(X) - \mu(D).$
The complement of $U \sqcup (X \setminus D)$ is $C$, so by \ref{usl2k}
$\,  \mu(U \sqcup (X \setminus D) ) = \mu(X) - \mu(C).$ 
Then $ \mu(U) + \mu(C) = \mu(D)$.  The case where $C \sqcup U$ is open can be proved similarly.
\end{proof}

Sometimes it is convenient to use the definition of a topological measure given through closed sets. The following result is from \cite[Sect. 2]{Wheeler}.

\begin{lemma} \label{tmCcx}
$\mu $ is a topological measure on $X$ iff $\mu$ is a real-valued, finite, nonnegative 
set function on $\ax$ satisfying the following properties:
\begin{enumerate}
\item
if $ C \se K, \ C,K \in \cx$ then $\mu(C) \le \mu(K).$ 
\item
if $C, K \in \cx$ are disjoint then $\mu(C) + \mu(K) = \mu(C \sc K).$ 
\item
if $C \in \cx$ and $ \eps >0$ then $ \exists K \in \cx$ such that $C \cap K= \O$ and 
$\mu(C) + \mu(K) > \mu(X) - \eps$.
\item
$\mu(U) = \mu(X) - \mu(X \sm U) $ for every  $ U \in \ox.$
\end{enumerate}
\end{lemma}

Thus, we can define a topological measure on $X$ by 
giving a set function on $\cx$ satisfying the first three conditions of Lemma \ref{tmCcx}, 
and then extend it to open sets using the last condition.   

\begin{remark} 
If  $X$ is compact and $ \mu$ is finite, condition \ref{T1}  of Definition \ref{tmC}  
is equivalent, as was noticed in \cite{Grubb:Signed}, to the following three conditions:
\begin{enumerate}[label=(\roman*),ref=(roman*)]
\item
$ \mu(U \sqcup V) = \mu(U) + \mu(V)$ for any two disjoint  open sets $U,V$.
\item
If $X= U \cup V, \ \  U , V \in \mathscr{O}(X)$ then $ \mu(U) + \mu(V) = \mu(X) + \mu(U \cap V)$.
\item
$\mu(X \setminus U) = \mu(X) - \mu(U)$ for any open set $U$.
\end{enumerate}
The same equivalence holds if we replace open sets by closed sets. 
Thus, when $X$ is compact,  a finite topological measure (and more generally, a  bounded signed topological measure) 
can be defined by its actions on open (respectively, on closed) sets. The idea of determining a topological measure on a closed manifold 
by its values on closed submanifolds with boundary is in \cite[Sect. 2]{Zap}. 
\end{remark}

\begin{remark} \label{Vloz}
A measure on $X$ is a countably additive set function on a $\sigma$-algebra of subsets of $X$ with values in $[0, \infty]$.  
A Borel measure on $X$ is a measure on the Borel $\sigma$-algebra on $X$.  
Let $X$ be locally compact, and let $ \mathscr{M}$  be the collection of all Borel measures on $X$ that are inner regular on open sets and outer regular 
on all Borel sets 
(i.e. $m(U) = \sup \{  m(K): K \subseteq U,\,  K  \text{  is compact} \}$ for every open set $U$, and 
$ m(E) = \inf \{ m(U): E \subseteq U, \, U \text{  is open} \} $ for every Borel set $E$). 
Thus, $  \mathscr{M}$ includes regular Borel measures and Radon measures. 
We denote by $M(X)$ the restrictions to $\mathscr{O}(X) \cup \mathscr{C}(X)$ of measures 
from $ \mathscr{M}$, and by $TM(X)$ the set of all topological measures on $X$.
Then
\begin{align*} 
 M(X) \subsetneqq  TM(X).
\end{align*}
The inclusions follow from the definitions. 
When $X$ is compact, there are examples of topological measures that are not measures in numerous papers, 
beginning with \cite{Aarnes:TheFirstPaper}. In Sections \ref{ExamplesC} and \ref{ExamplesLC} 
we give examples of topological measures that are not measures on 
compact and locally compact noncompact spaces.
\end{remark}

The next result tells when a topological measure is a Borel measure (\cite{Butler:DTMLC}).

\begin{theorem} \label{subaddit}
Let $\mu$ be a topological measure (or more generally, a deficient topological measure) on a locally compact space $X$. 
The following are equivalent: 
\begin{itemize}
\item[(a)]
If $C, K$ are compact subsets of $X$, then $\mu(C \cup K ) \le \mu(C) + \mu(K)$.
\item[(b)]
If $U, V$ are open subsets of $X$,  then $\mu(U \cup V) \le \mu(U) + \mu(V)$.
\item[(c)]
$\mu$ admits a unique extension to an inner regular on open sets, outer regular Borel measure 
$m$ on the Borel $\sigma$-algebra of subsets of $X$. 
$m$ is a Radon measure iff $\mu$ is compact-finite. 
If $\mu$ is finite then $m$ is an outer regular and inner closed regular Borel measure.
\end{itemize}
\end{theorem}

\section{Solid-set functions} \label{SSF}

Our goal now is to extend a set function defined on a smaller collection of subsets of $X$ than
$\mathscr{O}(X) \cup \mathscr{C}(X)$ to a topological measure on $X$. One such collection is the collection of solid
bounded open and solid compact sets, and the corresponding set function is a solid-set function.  

\begin{definition} \label{DeSSFLC}
A function $ \lambda: \mathscr{A}_{s}^{*}(X) \rightarrow [0, \infty) $ is a solid-set function on $X$ if
\begin{enumerate}[label=(s\arabic*),ref=(s\arabic*)]
\item \label{superadd}
$ \sum_{i=1}^n \lambda(C_i) \le \lambda(C)$ whenever $\bigsqcup_{i=1}^n C_i \subseteq C,  \ \  C, C_i \in \mathscr{K}_{s}(X)$; 
\item \label{regul}
$ \lambda(U) = \sup \{ \lambda(K): \ K \subseteq U , \ K \in \mathscr{K}_{s}(X) \}$ for $U \in \mathscr{O}_{s}^{*}(X)$; 
\item \label{regulo}
$ \lambda(K) = \inf \{ \lambda(U) : \  K \subseteq U, \ U \in \mathscr{O}_{s}^{*}(X) \}$ for $ K  \in \mathscr{K}_{s}(X)$; 
\item  \label{solidparti}
$ \lambda(A) = \sum_{i=1}^n \lambda (A_i)$ whenever $A = \bigsqcup_{i=1}^n A_i, \ \ A , A_i  \in \mathscr{A}_{s}^{*}(X)$.
\end{enumerate}
\end{definition}

\begin{lemma} \label{PropSsfLC}
Let $X$ be locally compact, connected,  locally connected. Suppose $\lambda$ is a solid-set function on $X$. Then
\begin{itemize}
\item[(i)]
$\lambda(\emptyset) = 0.$ 
\item[(ii)]
$\lambda$ is a superadditive set function, i.e.  
if $ \bigsqcup_{t \in T} A_t \subseteq A, $  where $A_t, A \in \mathscr{A}_{s}^{*}(X)$, then $\sum_{t \in T } \lambda(A_t)\le  \lambda(A).$
\end{itemize}
\end{lemma}

\begin{proof} 
From Definition \ref{DeSSFLC} we see that $\lambda(\emptyset) = 0$.
Now let  $ \bigsqcup_{t \in T} A_t \subseteq A, $  where $A_t, A \in \mathscr{A}_{s}^{*}(X)$.  
Since  $\sum_{t \in T } \lambda(A_t) = \sup \{ \sum_{t \in T'} \lambda(A_t) : \ T' \subseteq T,  \ T' \mbox{  is finite } \}$, 
it is enough to assume that $T$ is finite. By \ref{regul}  in Definition \ref{DeSSFLC} we may take all sets $A_t$ to be disjoint compact solid.
If also $A \in \mathscr{K}_{s}(X)$,  the assertion is just part \ref{superadd} of Definition \ref{DeSSFLC}. 
If $A \in \mathscr{O}_{s}^{*}(X)$ then  by Lemma \ref{LeCsInside} choose $ C \in \mathscr{K}_{s}(X)$ such that 
$\bigsqcup_{t \in T } A_t \subseteq C \subseteq A$. 
Now the assertion follows from parts  \ref{superadd} and \ref{regul}  of Definition \ref{DeSSFLC}.
\end{proof}


\section{Extension to $\mathscr{A}_{ss}^{*}(X) \cup \mathscr{K}_{c}(X)$} \label{ExtBssKc}

Now we assume that $X$ is locally compact, connected, locally connected.

\noindent
We start with a solid-set function $ \lambda: \mathscr{A}_{s}^{*}(X) \rightarrow [0, \infty)$ 
on a locally compact, connected, locally connected space $X$.
Our goal is to extend
$\lambda$ to a topological measure on $X$. We shall do this in steps, each time extending 
the current set function to a new set function defined on a larger collection of sets.

\begin{definition}  \label{la1LC}
For $A \in\mathscr{A}_{ss}^{*}(X) \cup \mathscr{K}_{c}(X)$ define 
$$  \lambda_1(A) = \lambda(\widetilde{A}) - \sum_{i \in I} \lambda(B_i)$$
where $\{ B_i : \  i \in I\} $ is the family of bounded components of $X \setminus A$.
In particular,  when $X$ is compact
\begin{align} \label{Xcomass}
  \lambda_1(A) = \lambda(X) - \sum_{i \in I} \lambda(B_i).
\end{align}
\end{definition}

\begin{remark} \label{subtmysl}
By Lemma \ref{SolidCompoLC} each $B_i \in \mathscr{A}_{s}^{*}(X)$.
If  $A \in\mathscr{A}_{ss}^{*}(X) \cup \mathscr{K}_{c}(X)$ 
then $\bigsqcup_{i \in I} B_i \subseteq \widetilde A$ and  by 
Lemma \ref{PropSsfLC} 
$\sum_{i \in I} \lambda(B_i) \le \lambda(\widetilde A).$
If  $A \in\mathscr{A}_{ss}^{*}(X)$ then in Definition \ref{la1LC} we subtract finitely many terms. If $ A \in \mathscr{K}_{c}(X)$ we use 
additivity on open solid sets (part \ref{smooth} in Lemma \ref{propTMLC}).
\end{remark} 

\begin{lemma} \label{Prla1LC}
The set function $\lambda_1: \mathscr{A}_{ss}^{*}(X) \cup \mathscr{K}_{c}(X) \rightarrow [0, \infty) $ defined in
Definition \ref{la1LC} satisfies the following properties:
\begin{enumerate}[label=(\roman*),ref=(\roman*)]
\item \label{pa1}
$\lambda_1$ is real-valued and $\lambda_1 = \lambda$ on  $\mathscr{A}_{s}^{*}(X).$
\item \label{pa2}
Suppose $\bigsqcup_{i=1}^n A_i  \sqcup \bigsqcup_{s \in S} B_s  \subseteq  A$, where 
$A, A_i \in \mathscr{A}_{ss}^{*}(X) \cup \mathscr{K}_{c}(X)$ and $B_s \in \mathscr{A}_{s}^{*}(X)$. Then 
$$ \sum_{i=1}^n \lambda_1(A_i) + \sum_{s \in S} \lambda_1(B_s) \le \lambda_1(A).$$
In particular, if $\bigsqcup_{i=1}^n C_i \subseteq C$ where $C_i, C \in \mathscr{K}_{c}(X)$ then 
$ \sum_{i=1}^n \lambda_1(C_i) \le \lambda_1(C)$
and  if $A \subseteq B, \ A,B \in \mathscr{A}_{ss}^{*}(X) \cup \mathscr{K}_{c}(X) $ then 
$\lambda_1(A) \le \lambda_1(B).$
\item \label{pa3}
Suppose that $\bigsqcup_{i=1}^n A_i  \sqcup \bigsqcup_{s \in S} B_s =  A$, where 
$A, A_i \in \mathscr{A}_{ss}^{*}(X) \cup \mathscr{K}_{c}(X)$ and $B_s \in \mathscr{A}_{s}^{*}(X)$
with at most finitely many of $B_s \in \mathscr{K}_{s}(X)$.
Then 
$ \sum_{i=1}^n \lambda_1(A_i) + \sum_{s \in S} \lambda_1(B_s) = \lambda_1(A).$
\end{enumerate}
\end{lemma}

\begin{proof}
(i)
Easy to see from  Lemma \ref{PrSolidHuLC}, Lemma \ref{PropSsfLC}, and Remark \ref{subtmysl}. \\
(ii)
First we assume that $X$ is noncompact. 
Suppose that $\bigsqcup_{i=1}^n A_i  \sqcup \bigsqcup_{s \in S} B_s  \subseteq  A$, where 
$A, A_i \in \mathscr{A}_{ss}^{*}(X) \cup \mathscr{K}_{c}(X)$ and $B_s \in \mathscr{A}_{s}^{*}(X)$.
We may assume that $A \in \mathscr{A}_{s}^{*}(X)$, since the inequality
\begin{eqnarray} \label{number1}
\sum_{i=1}^n \lambda_1(A_i) + \sum_{s \in S} \lambda_1(B_s) \le \lambda_1(A)
\end{eqnarray}
is equivalent to 
\begin{eqnarray} \label{number2} 
\sum_{i=1}^n \lambda_1(A_i) + \sum_{s \in S} \lambda_1(B_s) + 
\sum_{t \in T} \lambda_1 (D_t) \le \lambda_1(\widetilde{A}),
\end{eqnarray}
where $\{D_t: \  t \in T \}$ is the disjoint family of bounded components of $X \setminus A$, and
by Lemma \ref{SolidCompoLC} each $D_t \in \mathscr{A}_{s}^{*}(X)$.

The proof is by induction on $n$. 
For $n=0$ the statement is Lemma \ref{PropSsfLC}. 
Suppose now $n \ge 1$ and assume the result is true for any disjoint collection 
(contained in a bounded solid set)  of bounded semisolid or compact connected sets among which 
there are less than $n$ non-solid sets. Assume now that we have $n$ disjoint sets 
$A_1, \ldots, A_n$ from the collection
$\mathscr{A}_{ss}^{*}(X) \cup \mathscr{K}_{c}(X)$.
Consider a partial order on 
$\{ A_1, A_2, \ldots, A_n \}$ where $A_i \le A_j$ iff  $\widetilde{A_i} \subseteq \widetilde{A_j}$. 
(See Lemma \ref{PrSolidHuLC}.)
Let $ A_1, \ldots, A_p$ where $ p \le n$ be maximal elements in  $\{ A_1, A_2, \ldots A_n \}$
with respect to this partial order.
For a maximal element $A_k,  k \in \{ 1, \ldots, p\}$ 
define the following index sets:
$$ I_k = \{ i \in \{ p+1, \ldots, n\}: \,  A_i \mbox{ is contained in a bounded component of }  X\setminus A_k \},  $$
$$ S_k = \{ s \in S : \ B_s \mbox{ is contained in a bounded component of }  X\setminus A_k \}.  $$

Let $\{ E_{\alpha} \}_{\alpha \in H} $ be the disjoint family of bounded components 
of $X \setminus A_k$. Then  
$ I_k = \bigsqcup_{\alpha \in H} I_{k,\alpha}$, $ S_k = \bigsqcup_{\alpha \in H} S_{k,\alpha}  $
where 
$ I_{k, \alpha} = \{ i \in \{ p+1, \ldots, n\} :  \ A_i \subseteq E_\alpha \}$, $S_{k, \alpha} = \{ s \in S : \ B_s \subseteq E_\alpha \}. $
The set $I_k$ and each set $I_{k, \alpha}$ has cardinality $< n$. 
The set $E_\alpha$ is solid by Lemma \ref{SolidCompoLC}, and
\begin{eqnarray} \label{zv1}
\bigsqcup_{i \in I_{k, \alpha}} A_i \sqcup \bigsqcup_{s \in S_{k, \alpha}} B_s \subseteq E_\alpha. 
\end{eqnarray}
By induction hypothesis
$ \sum_{i \in I_{k, \alpha}} \lambda_1(A_i) +
\sum_{s \in S_{k, \alpha}} \lambda_1(B_s) \le \lambda_1(E_\alpha). $
It follows that 
\begin{align*}
 \sum_{i \in I_k} \lambda_1(A_i) + \sum_{s \in S_k} \lambda_1(B_s)
&= \sum_{\alpha \in H}\left( \sum_{i \in I_{k, \alpha}} \lambda_1(A_i) +
\sum_{s \in S_{k, \alpha}} \lambda_1(B_s) \right)  
\le \sum_{\alpha \in H} \lambda_1 (E_\alpha).
\end{align*}
Then using part \ref{pa1} and Definition \ref{la1LC} we have:
\begin{align} \label{Aktilde}
\lambda_1 (A_k) + \sum_{i \in I_k} \lambda_1(A_i) + \sum_{s \in S_k} \lambda_1(B_s) 
\le \lambda_1 (A_k) +\sum_{\alpha \in H} \lambda_1 (E_\alpha) =  \lambda_1(\widetilde{A_k}).
\end{align}

From Remark \ref{ordering} $ \{1, \ldots, n\} = \{1, \ldots, p\} \sqcup \bigsqcup_{k=1}^p I_k $. 
Similarly, the sets $S_k, \ k=1, \ldots, p$ are also disjoint.
Consider the index set $ S' = S \setminus \bigsqcup_{k=1}^p S_k.$
Since 
$ \{\widetilde{A_k}\}_{k=1}^p \bigsqcup \{B_s\}_{ s \in S'}$ is a collection of disjoint solid sets contained in the solid set $A$,
applying  (\ref{Aktilde}) and  Lemma \ref{PropSsfLC} we have:
\begin{eqnarray}
\sum_{i=1}^n \lambda_1(A_i)  &+& \sum_{s \in S} \lambda(B_s)  
= \sum_{k=1}^p \left( \lambda_1(A_k) + \sum_{i \in I_k} \lambda_1(A_i)  + 
 \sum_{s \in S_k} \lambda(B_s) \right) + \sum_{s \in S'} \lambda(B_s)  \nonumber \\
&\le& \sum_{k=1}^p \lambda(\widetilde{A_k}) + \sum_{s \in S'} \lambda(B_s) \le
 \lambda(A).
 \label{eqin1}
\end{eqnarray}
  
Now suppose $X$ is compact. 
As in (\ref{number1}) and (\ref{number2}), the inequality
$\sum_{i=1}^n \la_1(A_i) + \sum_{s \in S} \la_1(B_s) \le \la_1(A) $
is equivalent to 
$ \sum_{i=1}^n \la_1(A_i) + \sum_{s \in S} \la_1(B_s) + \sum_{t \in T} \la_1 (D_t) \le \la_1(X)$, 
where $\{D_t: \  t \in T \}$ is the disjoint family of components of $X \setminus A$, and each $D_t$ is a solid set. 
We need to show that 
$$\sum_{i=1}^n \la_1(A_i) + \sum_{s \in S} \la_1(B_s) \le \la_1(X) \ \mbox{  if   } \ \bigsqcup_{i=1}^n A_i  \sqcup \bigsqcup_{s \in S} B_s  \subseteq  X,$$
where each $B_s$ is solid and each $A_i \in \assx \cup \ccx$. 

The proof is by induction on $n$. For $n=0$ use Lemma \ref{PropSsfLC}.
Suppose  now $n\ge 1$ and  assume the result is true for any disjoint family from 
$\assx \cup \ccx \cup \asx$ that contains less than $n$ non-solid sets. 
Let 
$ X \sm A_1 = \bsc_{j \in J} E_j $
be the decomposition into components;  each $E_j$ is a solid set by Lemma \ref{SolidCompoLC}. 
Each $A_i, \ i\ge 2$ and each $B_s$ is contained in a component of $X \sm A_1$, so  
for each $j \in J$ define index sets
$ I_j = \{ i \in \{2, \ldots, n\}: \ A_i \se E_j  \}$, $ S_j = \{ s \in S: B_s \se E_j \}.$
Since $X \sm E_j$ is solid, $ | I_j| \le n-1$ and 
\begin{align} \label{vs1}
(X \sm E_j) \sc \bsc_{i \in I_j} A_i \sc \bsc_{s \in S_j} B_s  \se X 
\end{align}
by induction hypothesis we have
$  \la_1(X \sm E_j) + \sum_{ i \in I_j}  \la_1 (A_i) + \sum_{s \in S_j} \la_1 (B_s) \le \la_1(X)$, 
i.e. by Definition \ref{la1LC}
$\sum_{ i \in I_j}  \la_1 (A_i) + \sum_{s \in S_j} \la_1 (B_s) \le \la_1(E_j).$
Then we have:
\begin{eqnarray*}
\sum_{i=1}^n \la_1(A_i) &+& \sum_{s \in S} \la_1(B_s) 
= \la_1(A_1) + 
\sum_{j \in J} \left( \sum_{ i \in I_j}  \la_1 (A_i) + \sum_{s \in S_j} \la_1 (B_s) \right) \\
&\le& \la_1(A_1) + \sum_{j \in J} \la_1(E_j) = \la_1(X),
\end{eqnarray*} 
where the last equality is by definition of $\la_1$. \\
(iii)
The proof is almost identical to the proof of the previous part, and we keep the same notations. 
When $X$ is noncompact, we may assume again that  $A \in \mathscr{A}_{s}^{*}(X)$, 
since the inequalities (\ref{number1}) and (\ref{number2}) become equalities.
The proof is by induction on $n$, and the case $n=0$ is given by  
Lemma \ref{finiteSP} and part \ref{solidparti} of 
Definition \ref{DeSSFLC}. The inequalities in the induction step become equalities
once one observes that
$\bigsqcup_{k=1}^p \widetilde{A_k} \sqcup \bigsqcup_{s \in S'} B_s = A$ (see Remark \ref{ordering})
and that (\ref{zv1}) above becomes
$\bigsqcup_{i \in I_{k, \alpha}} A_i \sqcup \bigsqcup_{s \in S_{k , \alpha}} B_s =E_{\alpha}$.
The last inequality in (\ref{eqin1}) becomes an equality  by Lemma \ref{finiteSP}  and  part \ref{solidparti} of Definition \ref{DeSSFLC}. 
When $X$ is compact, the inclusion in (\ref{vs1}) becomes 
the equality, and then all subsequent inequalities in the proof of the previous part become equalities as well. 
\end{proof} 

\section{Extension to $\mathscr{K}_{0}(X)$} \label{BCOX}

Our goal now is to extend the set function $\lambda_1$ to a set function $\lambda_2$ 
defined on $\mathscr{K}_{0}(X).$  Recall that $ K \in \mathscr{K}_{0}(X)$ if $ K = \bigsqcup_{i=1}^n K_i$ where $ n \in \mathbb{N}$ and  
$K_i \in \mathscr{K}_{c}(X)$ for $ i=1, \ldots, n.$

\begin{definition}  \label{la2LC}
For $K = \bigsqcup_{i=1}^n K_i,$ where $ K_i \in \mathscr{K}_{c}(X)$, let 
$ \lambda_2(K ) = \sum_{i=1}^n \lambda_1(K_i). $ 
\end{definition}

\begin{lemma} \label{OssKo}
Suppose $C =C_1 \sc \ldots \sc C_n  \se U, C_i \in \ksx, U \in \bosx$. Given $\eps >0$, there exists $ B \se U \sm C,  \, B \in \kox$  such that
$ \la (U) - \la_2(C) - \la_2(B) < \eps$.
\end{lemma}

\begin{proof}
Let $ \eps >0$.The proof is by induction on $n$.
Let $n=1$, so $C = C_1$. Suppose $U \sm C$ is disconnected. By Lemma \ref{LeDecompV}  $ U = C \sc \bsc_{i=1}^m W_i$, where $W_i \in \bosx$. 
Pick $K_i \in \ksx$  such that 
$ \sum_{i=1}^m (\la(W_i) - \la(K_i)) < \eps$ and let 
$B = K_1 \sc \ldots \sc K_n$. 
Then 
$  \la(U) - \la(C) - \la_2(B) =  \sum_{i=1}^m (\la(W_i) - \la(K_i)) < \eps.$
Now suppose $U \sm C$ is connected. 
Using complete regularity of $X$ and Lemma  \ref{LeCsInside} for $\eps' = \eps/2$ 
pick $K, F  \in \ksx$ and $ W, V, O \in \osx$ such that $K \se W \se \cl W  \se U$, $C  \se O \se F \se V$, $V \se K$, $ \la( U) - \la(K) < \eps'$, and
$ \la(V) - \la(C) < \eps'.$  
Then also 
$$ \la(U) - \la(W) < \eps', \ \ \ \ \  \la(F) - \la(C) < \eps',$$
so we approximate $U$ and $C$ by $W$ and $F$.
By the above argument it is enough to assume $ W \setminus F$ is connected. Since $W$ and $F$ are solid sets and 
$X \sm (W \sm F) = (X \sm W) \sc F$, we see that $W \sm F$ is a bounded open semisolid set. 
If $X$ is compact, then 
$\la_1(W \sm F) = \la(X) - \la(X \sm W) - \la(F) = \la(W) - \la(F)$. 
If $X$ is noncompact, then $\widetilde{W \sm F} = W$, and   $\la_1(W \sm F) = \la(W) - \la(F)$. In any case,   
$$\la_1(W \sm F)  = \la(W) - \la(F).$$
Since $ W \setminus F$ is connected, so is $B = \overline{ W \setminus F} $. Thus, $B \in \mathscr{K}_c(X), \  B \se \cl W \sm O \se U \sm C$ and 
using part \ref{pa2} of Lemma \ref{Prla1LC}
\begin{align*}
 \la(U) - \la_2(C)- \la_1(B) &\le \la(U) - \la_2(C) -\la_1(W \sm F)  \\
 &= \la(U) - \la(C) - \la(W) + \la(F) < 2 \eps' = \eps.
\end{align*}
 
Now suppose that lemma holds for any open bounded open solid set which contains less than $n$ disjoint compact solid sets.
We shall show that the statement holds for $U \in \bosx$ containing  $C = C_1  \sc \ldots \sc C_n, \, C_i \in \ksx$.
Suppose that  $U \sm  C$ is disconnected.  By Lemma \ref{ossdec} write $ U = \bsc_{i \in I} C_i \sc  \bsc_{j=1}^p U_j$, where 
$ p \in \N$ and $ I \ne  \O$. We can write $\{ 1, \ldots, n\} = I \sc \bsc_{j=1}^p I_j $, where $I_j = \{ i: \, C_i \se U_j\}$.
Each $| I_j| < n$, so by induction hypothesis   
find $B_j  \se U_j \sm \bsc_{i \in I_j} C_i, \, B_j \in \mathscr{K}_{0}(X)$ such that $ \sum_{j=1}^p ( \la(U_i) - \sum_{j \in I_j} \la(C_j) - \la_2(B_j)) < \eps$. 
Set $B = \bsc_{j=1}^p B_j$. Then $B \se U \sm C$ and 
$$ \la(U) - \la_2 (C) - \la_2(B) = \la(U) - \sum_{i=1}^n \la(C_i) - \la_2(B)  
   = \sum_{j=1}^p  (\la(U_i) - \sum_{j \in I_j} \la(C_j) - \la_2 (B_j)) < \eps. $$
Now suppose  $U \sm  C$ is connected. 
As in the induction step $n=1$, choose
$K, F_i  \in \ksx$ and $ W, V_i, O_i \in \osx$ such that $K \se W \se \cl W \se U$, $C_i \se O_i  \se F_i \se V_i$, $V_i \se K$ for $i=1, \ldots, n$,
and $ \la(U) - \la(K) < \eps, \ \sum_{i=1}^n ( \la(V_i) - \la(C_i)) < \eps$.  Again, 
it is enough to assume $W \sm \bsc_{i=1}^n F_i$ is connected. Then 
$ B = \overline{W \sm \bsc_{i=1}^n F_i}$
is a compact connected set 
for which the argument similar to one in  step $n=1$ shows that $ B \se U \sm C$ and $\la(U) -  \la_2(C)  -\la_1(B) < \eps$.
\end{proof} 

\begin{lemma} \label{Lela2LC}
The set function $\lambda_2$ from Definition \ref{la2LC}  satisfies the following properties:
\begin{enumerate}[label=(\roman*),ref=(\roman*)]
\item
$\lambda_2$ is real-valued, $\lambda_2 = \lambda_1 $ on $\mathscr{K}_{c}(X)$ and $\lambda_2 = \lambda$ on $ \mathscr{K}_{s}(X)$. 
\item
$\lambda_2$ is finitely additive on $\mathscr{K}_{0}(X)$.
\item
$\lambda_2$ is monotone on $\mathscr{K}_{0}(X).$
\item \label{bossEq}
$\la_1(U) = \sup\{ \la_2 (K): K  \se U, \, K \in \bcox \} $ for $ U \in \bossx$.
\end{enumerate}
\end{lemma}

\begin{proof} 
Part (i) easily follows from the definition of $\lambda_2$ and 
Lemma \ref{Prla1LC}. Part (ii) is obvious. \\
(iii) Let $ C \subseteq K, $ where 
$C, \ K \in \mathscr{K}_{0}(X)$. Write $ C= \bigsqcup_{i=1}^n C_i, \ K = \bigsqcup_{j=1}^m K_j,$ where 
the sets $C_i (i=1, \ldots, n)$ and $ K_j (j =1, \ldots, m)$ are compact connected.
By connectedness, each $C_i$ is contained in one of the sets  $K_j.$ Consider index sets 
$ I_j  = \{ i : \ C_i \subseteq K_j \} $ 
for  $ j = 1, \ldots, m.$  By Lemma \ref{Prla1LC}  we have
$\sum_{i \in I_j} \lambda_1(C_i) \le \lambda_1(K_j).$
Then
$\lambda_2(C) = \sum_{i=1}^n \lambda_1(C_i)  = 
\sum_{j=1}^m \sum_{i \in I_j}  \lambda_1(C_i) 
\le \sum_{j=1}^m \lambda_1 (K_j) = \lambda_2 (K). $
(iv) If $K  \se U, \, K \in \bcox, \, U \in  \mathscr{O}_{ss}^*(X) $ then by part  \ref{pa2} of Lemma \ref{Prla1LC} $\la_2(K) \le \la_1(U)$.
Let $ \widetilde U = U \sc \bsc_{i=1}^n C_i$, where $C_i$ are bounded components of $ X \sm U$. By Lemma \ref{OssKo} find $B \in U, B \in \bcox$
such that $  \la(\widetilde U) - \sum_{i=1}^n \la_1(C_i) - \la_2(B) = \la_1(U) - \la_2(B) < \eps$. 
\end{proof}
  
\section{Extension to $\mathscr{O}(X) \cup \mathscr{C}(X)$} \label{ExttoTM}

We are now ready to extend the set function $\lambda_2$ to a set function $\mu$ 
defined on $\mathscr{O}(X) \cup \mathscr{C}(X).$

\begin{definition} \label{muLC}
For an open set $U$ and  a closed set $F$ we define 
$$ \mu(U)  = \sup\{ \lambda_2(K) : \ K \subseteq U , \ K \in \mathscr{K}_{0}(X) \}, $$  
$$ \mu(F) = \inf \{ \mu(U): \  F \subseteq U, \ U \in \mathscr{O}(X) \}.$$ 
\end{definition}

Note that $ \mu$ may assume $ \infty$.

\begin{lemma} \label{PropMuLC}
The set function $\mu$ in Definition \ref{muLC} satisfies the following properties:
\begin{enumerate}[label=(p\arabic*),ref=(p\arabic*)]
\item \label{monotoneLC}
$\mu$ is monotone, i.e. if $A \subseteq B, \, A,B \in \mathscr{O}(X) \cup \mathscr{C}(X)$ then $ \mu(A) \le \mu(B)$.
\item \label{finiteness}
$\mu(A) < \infty$ for each $ A \in \mathscr{A}^{*}(X)$. In particular, $ \mu$ is compact-finite.
\item \label{ineqla2}
$\mu \ge \lambda_2$ on $\mathscr{K}_{0}(X).$
\item \label{CoAppr}
Let $ K \subseteq V, K \in \mathscr{K}(X), \ V \in \mathscr{O}(X)$. Then for any positive $\epsilon$ there exists $ K_1 \in \mathscr{K}_{0}(X)$ 
such that 
$ K \subseteq K_1 \subseteq V$ and $ \mu(K_1) - \mu(K) < \epsilon.$
\item \label{extla2}
$\mu = \lambda$ on $ \mathscr{A}_{s}^{*}(X).$
\item\label{OpenFinAddLC}
$ \mu$ is finitely additive on open sets.
\item\label{CloFinAddLC}
If $G = F \sqcup K$, where $G, F \in \mathscr{C}(X), \ K \in \mathscr{K}(X)$ then $\mu(G) = \mu(F) + \mu(K).$ In particular, 
$\mu$ is finitely additive on compact sets.
\item \label{AddBox}
$\mu$ is additive on $\mathscr{O}(X)$, i.e.
if $V = \bigsqcup_{i \in I} V_i$, where $ V, \ V_i \in \mathscr{O}(X)$ for all $ i \in I$,  then
$\mu(V) = \sum_{i \in I} \mu(V_i). $
\item \label{superaddF}
If $G \sqcup V = F$ where $ G, F  \in \mathscr{C}(X), \ V \in \mathscr{O}(X)$ then $ \mu(G) + \mu(V) \le \mu(F).$
\item \label{superaddU}
If $G \sqcup V \subseteq U$ where $ G  \in \mathscr{C}(X),  \ V,U  \in \mathscr{O}(X)$ then $ \mu(G) + \mu(V) \le \mu(U).$
\item \label{mula1}
$\mu = \lambda_1$ on $\mathscr{K}_{c}(X) \cup \mathscr{O}_{ss}^{*}(X) $ and $\mu = \lambda_2$ on $\mathscr{K}_{0}(X)$.
\item \label{regularityLC}
$\mu(U) = \sup\{\mu(C): \ C \subseteq U , \ C \in \mathscr{K}(X) \}$ for $ U  \in \mathscr{O}(X).$
\end{enumerate}
\end{lemma}

\begin{proof}
(p1)
Clearly, $\mu$ is monotone on open sets and on closed sets.
Let $ V \in \mathscr{O}(X), F \in \mathscr{C}(X)$. The monotonicity in the case $ F \subseteq V$ is obvious.
Suppose $ V \subseteq F$. For any open set $U $ with $ F \subseteq U$ we have $V \subseteq U$, so $ \mu(V) \le \mu(U)$, 
Taking infimum over sets $U$ we obtain $ \mu(V) \le \mu(F)$.   \\
(p2)
Let $K \in \mathscr{K}(X)$. By Lemma \ref{LeConLC}  choose $V \in \mathscr{O}_{c}^{*}(X)$  and $C \in \mathscr{K}_{c}(X)$ such that 
$ K \subseteq V \subseteq C \subseteq X$. For any $D \in \mathscr{K}_{0}(X), D \subseteq V$ by Lemma \ref{Lela2LC} we have $\lambda_2(D) \le \lambda_2(C)$, and 
$\lambda_2(C) < \infty$. 
By Definition \ref{muLC} $\mu(V) \le \lambda_2(C)$, and then $ \mu(K) \le \mu(V) \le \lambda_2(C) < \infty$. 
Thus, $ \mu$ is compact-finite.
If $U$ is an open bounded set then $ \mu(U) \le \mu(\overline{U}) < \infty$.  \\
(p3)
Let $K \in \mathscr{K}_{0}(X).$ For any open set $U$ containing $K$  we have 
$\mu(U) \ge \lambda_2(K)$ by the definition of $\mu$. Taking infimum over sets $U$ we obtain $\mu(K)  \ge \lambda_2(K).$  \\
(p4)
$\mu(K) < \infty$, so by Definition \ref{muLC} find $U \in \mathscr{O}(X)$ such that $ U \subseteq V,  \, \mu(U)  - \mu(K) < \epsilon$. 
Let $U_1, \ldots, U_n$ be 
finitely many connected components of $U$ that cover $K$. 
By Lemma \ref{LeConLC} 
 pick $V_i \in \mathscr{O}_{c}^{*}(X)$  such that
$K \cap U_i \subseteq V_i \subseteq \overline{ V_i} \subseteq U_i$ for $ i =1, \ldots, n$. We may take 
$K_1 = \bigsqcup_{i=1}^n \overline{ V_i}$, for $ K_1 \subseteq V$ and 
$ \mu(K_1) - \mu(K) < \mu(\bigsqcup_{i=1}^n U_i) - \mu(K) \le \mu(U) - \mu(K) < \epsilon .$ \\
(p5)
First we shall show that $\mu= \lambda$ on $\mathscr{O}_{s}^{*}(X)$. Let $U \in \mathscr{O}_{s}^{*}(X)$, 
so by part \ref{finiteness} $ \mu(U) < \infty$. 
By Definition \ref{muLC}, given $\epsilon > 0$, choose $K \in \mathscr{K}_{0}(X)$ such that
$K \subseteq U$ and $ \mu(U) -\epsilon < \lambda_2(K)$.
By Lemma \ref{LeCsInside} there exists  $ C \in \mathscr{K}_{s}(X)$ such that $ K \subseteq C \subseteq U$.
Now using Lemma \ref{Lela2LC} and Definition \ref{DeSSFLC} we have:
$\mu(U) - \epsilon < \lambda_2(K)  \le \lambda_2(C) = \la(C) 
\le \sup \{ \lambda(C) : \ \ C \subseteq U, \ C \in \mathscr{K}_{s}(X)\} = \lambda(U).$
Hence, $\mu(U) \le \lambda(U)$.
For the opposite inequality,  by Lemma \ref{Lela2LC}
and Definition \ref{DeSSFLC}
\begin{align*}
\lambda(U) &=  \sup\{ \lambda(C): \ C \subseteq U,  C \in \mathscr{K}_{s}(X) \} 
= \sup\{ \lambda_2 (C): \ C \subseteq U,  C \in \mathscr{K}_{s}(X) \} \\  
&\le \sup\{ \lambda_2 (C): \  C \subseteq U,  C \in \mathscr{K}_{0}(X) \} = \mu(U).
\end{align*}
Therefore, $\mu(U) = \lambda(U)$ for any $ U \in \mathscr{O}_{s}^{*}(X)$. 
Now we shall show that $\mu = \lambda$ on $\mathscr{K}_{s}(X)$.  From part \ref{ineqla2} above and Lemma \ref{Lela2LC}
we have $\mu \ge \lambda_2 = \lambda$ on $\mathscr{K}_{s}(X)$. 
Since $\mu = \lambda$ on $\mathscr{O}_{s}^{*}(X)$, for $C \in \mathscr{K}_{s}(X)$ we have by Definition \ref{DeSSFLC} and 
Definition \ref{muLC}:
\begin{align*}
\lambda(C) &=  \inf \{ \lambda(U): \ \ U \in \mathscr{O}_{s}^{*}(X), \ C \subseteq U \}  
= \inf \{ \mu(U): \ \ U \in \mathscr{O}_{s}^{*}(X), \ C \subseteq U \} \\
&\ge \inf \{ \mu(U): \ \ U \in \mathscr{O}(X) , \ C \subseteq U \}  = \mu(C).
\end{align*}
Therefore, $\mu = \lambda $ on $\mathscr{K}_{s}(X)$.  \\
(p6)
Let $U_1 , U_2 \in \mathscr{O}(X)$ be disjoint. For any $C_i, C_2 \in \mathscr{K}_{0}(X)$ 
with $ C_i \subseteq U_i, \ i=1,2$  by Lemma \ref{Lela2LC} and Definition \ref{muLC} we have
$ \lambda_2(C_1)  + \lambda_2(C_2) = \lambda_2(C_1 \sqcup C_2) \le \mu(U_1 \sqcup U_2).$
Then by Definition  \ref{muLC} we obtain
$ \mu(U_1) + \mu(U_2) \le \mu(U_1 \sqcup U_2).$ 
For the converse inequality, note that given $ C \subseteq U_1 \sqcup U_2, \  C \in \mathscr{K}_{0}(X)$ we 
have   $C = C_1 \sqcup C_2$, where $C_i = C \cap U_i \in \mathscr{K}_{0}(X), \ i =1,2$ (since each connected component of $C$ must be 
contained either  in $U_1$ or in $U_2$). Then
$ \lambda_2(C) = \lambda_2(C_1) + \lambda_2(C_2) \le \mu(U_1)  + \mu(U_2),$
giving
$ \mu(U_1 \sqcup U_2)  \le  \mu(U_1)  + \mu(U_2).$ \\
(p7)
A  compact and a closed set that are disjoint can be separated by disjoint open sets, and we can use the argument from part  \ref{compclo}  of 
Proposition \ref{PrFinAddLC}.  \\
(p8)
Let $ V = \bigsqcup_{i \in I} V_i$ with $V, V_i \in \mathscr{O}(X)$ for all $ i \in I$.
By parts \ref{OpenFinAddLC} and  \ref{monotoneLC}
for any finite $I' \subseteq I$  we have
$ \sum_{i \in I'} \mu(V_i) = \mu(\bigsqcup_{i \in I'} V_i) \le \mu(V), $ so
$ \sum_{i \in I} \mu(V_i) \le \mu(V).$
To prove the opposite inequality, first assume that $ \mu(V) < \infty$.
For $\epsilon >0 $ find a compact $C \in \mathscr{K}_{0}(X)$ contained in $V$ 
such that $ \mu(V) - \epsilon < \lambda_2(C).$ By compactness, $ C \subseteq \bigsqcup_{i \in I'} V_i$ 
for some finite subset $I'$ of $I$.
Then by connectedness $C = \bigsqcup_{i \in I'} C_i$ where 
$C_i = C \cap V_i \subseteq V_i$, and $C_i \in \mathscr{K}_{0}(X)$  for each $i \in I'.$ 
By Lemma \ref{Lela2LC} and part \ref{ineqla2} we have:
$$\mu(V)-\epsilon < \lambda_2(C) = \lambda_2(\bigsqcup_{i \in I'} C_i) =  \sum_{i \in I'} \lambda_2(C_i) \le  \sum_{i \in I'} \mu(C_i)  
\le  \sum_{i \in I'} \mu(V_i) \le \sum_{i \in I} \mu(V_i). $$ 
This shows that  $\mu(V) = \sum_{i \in I} \mu(V_i)$ when $ \mu(V) < \infty$.

Now suppose $ \mu(V) = \infty$. For $ n \in \mathbb{N}$ find a compact $ K \subseteq V$ such that $ \mu(K) > n$.
Choose a finite index set $I_n \subseteq I$ such that $ K \subseteq \bigsqcup_{i \in I_n} V_i$. Then
$ \sum_{i \in I} \mu(V_i) \ge  \sum_{i \in I_n} \mu(V_i) = \mu (\bigsqcup_{i \in I_n} V_i) \ge \mu(K) > n.$
Thus $ \sum_{i \in I} \mu(V_i) = \infty = \mu(V)$. \\
(p9)
It is  enough to show the statement for the case $ \mu(F) < \infty$.
If $K  \subseteq V, \  K \in \mathscr{K}_{0}(X)$ then $ G \sqcup K \subseteq F$. By parts \ref{ineqla2}, \ref{CloFinAddLC} and 
\ref{monotoneLC} $\  \mu(G) + \lambda_2(K) \le \mu(G) + \mu(K) \le \mu(F)$.  
Then $ \mu(G) + \mu(V) \le \mu(F)$. \\
(p10)
It is  enough to show the statement for the case $ \mu(U) < \infty$.
If $K  \subseteq V, \  K \in \mathscr{K}_{0}(X)$ then $F= G \sqcup K \subseteq U$. By  parts \ref{ineqla2}, \ref{CloFinAddLC}, 
and  Definition \ref{muLC} $\mu(G) + \lambda_2(K) \le \mu(G) + \mu(K) = \mu(F) \le \mu(U)$. 
Then  $ \mu(G) + \mu(V) \le \mu(U)$. \\
(p11)
Let $C \in \mathscr{K}_{c}(X)$.  According to Lemma \ref{SolidCompoLC} and 
Definition \ref{solid hull} write $\widetilde C \in \mathscr{K}_{s}(X)$ as  
$ \widetilde C = C \sqcup \bigsqcup_{i \in I} U_i$
where $U_i \in \mathscr{O}_{s}^{*}(X)$ are the bounded components of $X \setminus C$.
Given $\epsilon>0$ choose by Definition \ref{DeSSFLC} $V \in \mathscr{O}_{s}^{*}(X)$ such that 
$\widetilde C \subseteq V$ and $ \lambda(V) -\lambda(\widetilde C) < \epsilon$. 
By parts  \ref{AddBox}, \ref{superaddF}, and \ref{monotoneLC}
$$ \mu(C) + \sum_{i\in I} \mu(U_i) =  \mu(C) + \mu (\bigsqcup_{i\in I} (U_i))  \le \mu(\widetilde C) \le \mu(V).$$
Then using part \ref{extla2} and Definition \ref{la1LC} we have:
$$ \mu(C) \le \mu(V) - \sum_{i \in I} \mu(U_i) = \lambda(V) - \sum_{i\in I} \lambda(U_i) 
\le \lambda(\widetilde C)  - \sum_{i\in I} \lambda(U_i) + \epsilon = \lambda_1(C) + \epsilon.$$
Thus, $\mu(C) \le \lambda_1(C)$. By part \ref{ineqla2} and Lemma \ref{Lela2LC}
$\mu(C) \ge \lambda_2(C) =\lambda_1(C)$. So $\mu =\lambda_1$ on $\mathscr{K}_{c}(X)$. 
By part \ref{bossEq} of Lemma \ref{Lela2LC}  $\mu =\lambda_1$ on $\mathscr{O}_{ss}^*(X)$. 
From part \ref{CloFinAddLC} and Definition \ref{la2LC} we have $\mu = \lambda_2$ on $\mathscr{K}_{0}(X)$. \\ 
(p12)
Using part  \ref{ineqla2}
\begin{align*}
\mu(U) =& \sup \{\lambda_2(C) : C \subseteq U , \ C  \in \mathscr{K}_{0}(X) \} 
\le  \sup \{\mu(C) : C \subseteq U , \ C  \in \mathscr{K}_{0}(X) \} \\ 
&\le  \sup \{\mu(C) : C \subseteq U , \ C  \in \mathscr{K}(X) \}. 
\end{align*}
For the converse inequality, given $ C \subseteq U, \  U \in \mathscr{O}(X), \ C \in \mathscr{K}(X)$ choose 
by Lemma \ref{LeCCoU}   $K \in \mathscr{K}_{0}(X)$ with $ C \subseteq K \subseteq U$. 
Then  by parts \ref{monotoneLC} and \ref{mula1} $\mu(C) \le \mu(K) = \lambda_2(K)$, so 
$$ \sup \{\mu(C) : C \subseteq U , \ C  \in \mathscr{K}(X) \}  \le  \sup \{\lambda_2(K) : K \subseteq U , \ K  \in \mathscr{K}_{0}(X) \}   
= \mu(U). $$
\end{proof}

\section{Finite additivity on $  \mathscr{O}(X) \cup  \mathscr{K}(X)$} \label{finAdAll}

Finite additivity of $ \mu$ (defined in Definition \ref{muLC}) on $  \mathscr{O}(X) \cup  \mathscr{K}(X)$ will be established in a series of lemmas.

\begin{lemma} \label{step2a}
If $C \subseteq U, \ C \in \mathscr{K}_{0}(X), \ U \in  \mathscr{O}_{ss}^{*}(X)$ then 
$ \mu(U) = \mu(C) + \mu(U \setminus C). $
\end{lemma}

\begin{proof}
Let $C =  C_1 \sqcup C_2 \sqcup \ldots \sqcup C_n$.
By Lemma \ref{DecUssN} 
$ U = \bsc_{j=1}^n C_j  \sqcup \bigsqcup_{t \in T}  U_t $,
where sets $U_t \in  \mathscr{O}_{ss}^{*}(X)$, and all but finitely many are in $\mathscr{O}_{s}^{*}(X)$.
By part \ref{pa3} of Lemma \ref{Prla1LC}
$ \la(U) = \sum_{i=1}^n \la_1(C_i) + \sum_{t \in T} \la_1 (U_t),$
so by parts \ref{mula1} and \ref{AddBox} of Lemma \ref{PropMuLC} 
$ \mu(U) - \mu(C) = \sum_{t \in T} \mu (U_t)  = \mu( U \setminus C). $
\end{proof}

\begin{lemma} \label{AddXcomp}
If $X$ is compact, $K \subseteq X, \ K  \in \mathscr{C}(X), \ U \in  \mathscr{O}(X)$  then 
$ \mu(X) = \mu(K) + \mu(X \setminus K). $
\end{lemma}
 
\begin{proof}
Since $X$ is solid, by part \ref{extla2} of Lemma \ref{PropMuLC} $ \mu$ is finite. 
Let $K  \in \mathscr{C}(X)$. Given $ \eps >0$ by part \ref{CoAppr} choose $K_1 \in  \mathscr{K}_{0}(X), K  \se K_1$ 
such that $\mu(K) > \mu(K_1) - \eps$.
Using Lemma \ref{step2a} we have:
$ \mu(K) + \mu(X \sm K) > \mu(K_1) - \eps + \mu(X \sm K) \ge \mu(K_1)  + \mu(X \sm K_1) - \eps = \mu(X) - \eps.$
Thus, $ \mu(K) + \mu(X \sm K)  \ge \mu(X)$. The opposite inequality is by part \ref{superaddF} of Lemma \ref{PropMuLC}.
\end{proof}

\begin{lemma} \label{step3}
If $X$ is noncompact, $K \subseteq U, \ K \in \mathscr{K}(X), \ U \in  \mathscr{O}_{c}^{*}(X)$ then  
$ \mu(U) = \mu(K) + \mu(U \setminus K). $
\end{lemma}

\begin{proof}
Using part \ref{regularityLC} of Lemma \ref{PropMuLC} and  Lemma \ref{LeConLC}  
choose sets $ W \in \mathscr{O}_{c}^{*}(X)$ and $ D \in \mathscr{K}_{c}(X)$ such that 
\begin{align} \label{WD}
 K \subseteq W \subseteq D \subseteq U \mbox{   and   } \mu(U) - \mu(W) < \epsilon. 
\end{align}
Let $B$ be the union 
of bounded components of $ X \setminus U$ and let the open set $V$ 
be the union of bounded components of $ X \setminus D$. Set 
$C = B \cap V.$
By Lemma \ref{LeCleverSet}
$C$ is compact and $ U \sqcup C$ is open.
The solid hull $\widetilde D= D  \sqcup V$. 
Then by part \ref{superaddF} of Lemma \ref{PropMuLC}
$\mu(D) + \mu(V) \le \mu(\widetilde D)$. 
By Lemma \ref{PrSolidHuLC}
$V \subseteq \widetilde D  \subseteq \widetilde U  = U \sqcup B$, so 
$ V \subseteq  U \sqcup (B \cap V) = U \sqcup C. $
It follows that
$ K \sqcup C \subseteq D \sqcup V = \widetilde D \subseteq U \sqcup C. $
Since $U \sqcup C$ is open, by Lemma \ref{LeCsInside}  we may find $ W' \in \mathscr{O}_{s}^{*}(X)$ such that 
\begin{eqnarray} \label{sha} 
K \sqcup C \subseteq \widetilde D \subseteq W' \subseteq U \sqcup C.
\end{eqnarray}
Then 
\begin{eqnarray} \label{W'} 
W' \setminus(K \sqcup C) \subseteq U \setminus K.
\end{eqnarray}
According to part \ref{CoAppr} of Lemma \ref{PropMuLC}, pick $K_1 \in \mathscr{K}_{0}(X)$ such that 
\begin{eqnarray} \label{K1}
K \sqcup C \subseteq K_1 \subseteq W'  \mbox{   and    } \mu(K_1) \le \mu(K \sqcup C) + \epsilon.
\end{eqnarray}
By Lemma \ref{step2a}, $\mu(W') = \mu(W' \setminus K_1) + \mu(K_1)$.
Now using (\ref{WD}), Definition \ref{la1LC}, 
(\ref{sha}),  (\ref{K1}), (\ref{W'}), additivity on $\mathscr{O}(X)$ and finite additivity of 
$\mu$ on $\mathscr{K}(X)$ in Lemma \ref{PropMuLC}  we have:  
\begin{align*}
\mu(U) - \epsilon  &< \mu(W) \le \mu(D)= \mu(\widetilde D) - \mu(V) 
\le  \mu(W') - \mu(C)  \\
&= \mu(W' \setminus K_1) + \mu(K_1) - \mu(C) 
\le \mu(W'  \setminus (K \sqcup C)) + \mu(K \sqcup C) + \epsilon - \mu(C) \\
&\le \mu(U \setminus K)  + \mu(K) + \mu(C) -\mu(C) + \epsilon
=\mu(U \setminus K)  + \mu(K)  + \epsilon.
\end{align*}
Thus, $\mu(U)  \le \mu(U \setminus K)  + \mu(K)$. 
The opposite inequality is by part \ref{superaddU} of Lemma \ref{PropMuLC}. 
\end{proof}

\begin{lemma} \label{step4}
If $K \subseteq U, \ K \in \mathscr{K}(X), \ U \in  \mathscr{O}^{*}(X)$ then
$ \mu(U) = \mu(K) + \mu(U \setminus K). $
\end{lemma}

\begin{proof}
Let $ U = \bigsqcup_{i \in I} U_i$ be the decomposition of $U$ into connected components, and 
let $I'$ be a finite subset of $I$ such that $ K \subseteq \bigsqcup_{i \in I'} U_i$. For $i \in I'$ let
$K_i = K \cap U_i  \in \mathscr{K}(X)$ and let $K = \bigsqcup_{i \in I'} K_i$.
By Lemma \ref{step3}  we know that 
\begin{eqnarray} \label{A} 
\mu(K_i)  + \mu(U_i \setminus K_i) = \mu(U_i) \ \ \ \ \   \mbox{for  each} \ \  i \in I'.
\end{eqnarray}
Then  using finite additivity of $\mu$ on $\mathscr{K}(X)$ and additivity of $\mu$ on 
$\mathscr{O}(X)$ in Lemma \ref{PropMuLC}, and (\ref{A})  we have:
\begin{align*}
\mu(K) &+ \mu(U \setminus K) = \mu(\bigsqcup_{i \in I'}  K_i)  + \mu (U \setminus \bigsqcup_{i \in I'}  K_i) 
= \sum_{i \in I'} \mu(K_i) + \sum_{i \in I'} \mu(U_i \setminus K_i) + \sum_{i \in I \setminus I'} \mu(U_i) \\
&= \sum_{i \in I'} \mu(U_i) + \sum_{i \in I \setminus I'} \mu(U_i) = \sum_{i \in I} \mu(U_i) = \mu(U).
\end{align*} 
\end{proof} 

\begin{lemma} \label{step5}
If $K \subseteq U, \ K  \in \mathscr{K}(X), \ U \in  \mathscr{O}(X)$  then
$ \mu(U) = \mu(K) + \mu(U \setminus K). $
\end{lemma}
 
\begin{proof}
First assume that $ \mu(U) < \infty$.
Given $\epsilon >0$ by Definition \ref{muLC} find $C \in \mathscr{K}(X) $ such that $ K \subseteq C$ and $\mu(U) - \mu(C) < \epsilon $.
Using  Lemma \ref{easyLeLC} find 
$ V \in \mathscr{O}^{*}(X)$  such that 
$ K \subseteq C \subseteq V \subseteq U.$ 
By Lemma \ref{step4} $ \mu(V) = \mu(V \setminus K)  + \mu(K)$.
Then using monotonicity of $\mu$ in Lemma \ref{PropMuLC} we see that
\begin{align} \label{muCVK}
  \mu(U) - \epsilon <  \mu(C) \le \mu(V) =  \mu(V \setminus K)  + \mu(K) \le \mu(U \setminus K) + \mu(K).
\end{align}
Therefore, $ \mu(U)  \le \mu(U \setminus K) + \mu(K)$.
The opposite inequality is part  \ref{superaddU} of Lemma \ref{PropMuLC}. 
Therefore,
$ \mu(U)  = \mu(U \setminus K) + \mu(K)$ if $ \mu(U) < \infty.$
Now assume $ \mu(U) = \infty$.
For $n \in \mathbb{N}$ choose $ C \in \mathscr{K}(X)$ such that $ K \subseteq C$ and $ \mu(C) > n$.
By Lemma \ref{easyLeLC} find $ V \in \mathscr{O}^{*}(X)$  such that 
$ K \subseteq C \subseteq V \subseteq U.$
Using again (\ref{muCVK}) we have: 
$ n < \mu(C) \le \mu(V \setminus K) + \mu(K) \le \mu(U \setminus K) + \mu(K),  $
i.e. $n - \mu(K) \le \mu(U \setminus K)$.
Since $ \mu(K) \in \mathbb{R}$ by part \ref{finiteness} of Lemma \ref{PropMuLC}, 
it follows that $ \mu (U \setminus K) = \infty$, and $ \mu(U \setminus K) + \mu(K) = \mu(U)$.
\end{proof}

\begin{remark}
Our proof of Lemma \ref{step3} is close to that of \cite[Lemma 5.9]{Aarnes:LC}. We would like to point out that the part related to 
Lemma \ref{step2a} is the lengthiest and technically most difficult in the entire \cite{Aarnes:LC} as  
in involves  \cite[Lemma 5.8, Proposition 4.1, Lemma 4.3]{Aarnes:LC}
as well as lengthy adaptations of arguments from \cite{Aarnes:ConstructionPaper}.
\end{remark}

\begin{theorem} \label{extThLC}
Let $X$ be locally compact, connected, locally connected.
A solid-set function on $X$ extends uniquely to a compact-finite topological measure on $X$.
\end{theorem}

\begin{proof}
Definitions \ref{la1LC}, \ref{la2LC} and \ref{muLC} extend solid-set function 
$\lambda$ to a set function $\mu$. We shall show that $\mu$ is a topological measure.
Definition \ref{muLC} and part \ref{regularityLC} of  Lemma \ref{PropMuLC} 
show that $\mu$ satisfies \ref{TM2}  and \ref{TM3} of Definition \ref{TMLC}.
Proposition \ref{PrFinAddLC}, part  \ref{OpenFinAddLC} of Lemma \ref{PropMuLC}, and Lemma \ref{step5} 
show that $\mu$ is a topological measure if $X$ is noncompact. 
Proposition \ref{PrFinAddC},  part  \ref{OpenFinAddLC} of Lemma \ref{PropMuLC}, and Lemma \ref{AddXcomp}
show that $\mu$ is a topological measure if $X$ is compact. 
By part \ref{finiteness} of Lemma  \ref{PropMuLC} $\mu$ is compact-finite. 
To show the uniqueness of the extension suppose $\nu$ is a topological measure on $X$ such that $ \mu = \nu = \lambda$ on $ \mathscr{A}_{s}^{*}(X)$.
If $A \in \mathscr{K}_{c}(X)$ then by Definition \ref{solid hull} $A = \widetilde A \setminus (\bigsqcup_{s \in S} B_s)$, 
where $ \widetilde A, B_s \in \mathscr{A}_{s}^{*}(X)$, 
so from Definition \ref{la1LC} it follows that $\mu = \nu$ on  $ \mathscr{K}_{c}(X)$, and, hence, on $\mathscr{K}_{0}(X)$. 
From part \ref{CoRegulLC} of Lemma \ref{propTMLC} it then follows that $\mu = \nu$ on $\mathscr{O}(X)$, so $ \mu = \nu$.
\end{proof} 

\begin{remark} \label{extsumme}
We will summarize the extension procedure for obtaining a topological measure $\mu$ from a
solid-set function $\lambda$ on a locally compact, connected, locally connected space.
First, for a compact connected set $C$ we have:
$$ \mu(C) = \lambda(\widetilde C) - \sum_{i=1}^n \lambda(B_i), $$
where $\widetilde C$ is the solid hull of $C$ and $B_i$ (open solid sets) are bounded components of $X \setminus C$.  
Hence, when $X$ is compact  and $C$ is a closed connected set
$$ \mu(C) = \lambda(X) - \sum_{i=1}^n \lambda(B_i). $$
For $C \in  \mathscr{K}_{0}(X)$ of the form $C = \bsc_{i=1}^n C_i, \, C_i  \in  \mathscr{K}_{c}^{*}(X)$ 
we have:
$$ \mu (C) = \sum_{i=1}^n  \mu(C_i). $$
Finally,
$$ \mu(U)  = \sup\{ \mu(K) : \ K \subseteq U , \ K \in \mathscr{K}_{0}(X) \} \ \ \ \ \mbox{for an open set} \  U,$$ 
$$ \mu(F) = \inf \{ \mu(U): \  F \subseteq U, \ U \in \mathscr{O}(X) \} \ \ \ \ \mbox{for a closed set} \  F.$$ 
\end{remark}

\begin{theorem} \label{Tpart2}
If a solid-set function $\lambda$ is extended to a topological measure $\mu$ 
then the following holds: if $\lambda: \mathscr{A}_{s}^{*}(X) \rightarrow \{0,1\}$ then $\mu$  is also simple;
if $\sup \{ \lambda(K): \ K \in \mathscr{K}_{s}(X)\} = M < \infty$ then $\mu$ is finite and $ \mu(X) = M.$
\end{theorem}

\begin{proof}
Follows from Remark \ref{extsumme}, part \ref{extla2} of Lemma \ref{PropMuLC}, and part \ref{l8}  of Lemma \ref{propTMLC}. 
\end{proof}

\begin{theorem} \label{ExtUniq}
The restriction $\lambda$ of a compact-finite topological measure $\mu$ to $\mathscr{A}_{s}^{*}(X)$ is a solid-set function, and
$\mu$ is uniquely determined by $\lambda$. 
\end{theorem}

\begin{proof}
Let $\lambda$ be the restriction of $\mu$ to $ \mathscr{A}_{s}^{*}(X)$. 
Monotonicity of a topological measure (see  Lemma \ref{propTMLC}) and \ref{TM1} of Definition  \ref{TMLC}
show that $\lambda$ satisfies conditions \ref{superadd} and \ref{solidparti}
of Definition  \ref{DeSSFLC}. 
For $ U \in \mathscr{O}_{s}^{*}(X)$ by \ref{TM2} let $ K \in \mathscr{K}(X)$ be such that $\mu(U) - \mu(K) < \epsilon$ and by Lemma \ref{LeCsInside}
we may assume that $K \in \mathscr{K}_{s}(X)$. 
Part \ref{regul} of Definition  \ref{DeSSFLC} follows.  
Part  \ref{regulo} of Definition  \ref{DeSSFLC} follows from \ref{TM3} and Lemma \ref{LeCsInside}. 
Since $\mu$ is compact-finite, $\lambda$ is real-valued. Therefore, $\lambda$ is a solid-set function.
\end{proof}

\begin{remark}
When $X$ is compact our method of extending a solid-set function to a topological measure is different and simpler
than the original method in \cite{Aarnes:ConstructionPaper}.
\end{remark}

\begin{remark} \label{additionalPropMu}
Comparing values on semisolid or solid sets one may determine whether 
two topological measures (or two quasi-integrals) coincide. 
Lemma  \ref{PropSsfLC}, Lemma \ref{Prla1LC},  and Lemma \ref{PropMuLC} give us 
some additional properties of topological measures. For 
example, by part \ref{CloFinAddLC} of Lemma \ref{PropMuLC}, if a closed set $F$ and  a compact $K$ are disjoint,
then $\mu(F \sqcup K) = \mu(F)  + \mu(K)$. By part \ref{pa3} of  Lemma \ref{Prla1LC}, if 
$\bigsqcup_{i=1}^n A_i  \sqcup \bigsqcup_{s \in S} B_s =  A$, where 
$A, A_i \in \mathscr{A}_{ss}^{*}(X) \cup \mathscr{K}_{c}(X)$ 
and $B_s \in \mathscr{A}_{s}^{*}(X)$ with at most finitely many of $B_s \in \mathscr{K}_{s}(X)$,
then 
$ \sum_{i=1}^n \mu(A_i) + \sum_{s \in S} \mu(B_s) = \mu(A).$
\end{remark} 
 
\section{Irreducible partitions and genus of a space} \label{IrrPartInfo}

Let $X$ be a compact Hausdorff connected locally connected space.  $\asx = \osx \cup \csx$.

\begin{definition} \label{SsetPar}
A solid partition of a solid set $A$ is a finite disjoint collection of sets $ \{ A_i: \in I, \ A_i \in \asx \} $ such that $A = \bsc_{i \in I} A_i$.
\end{definition}
     
\begin{remark}                                                                                                                                                                                                                                                                                                                                                                                                                                                                                                                                                                                                                                                                                                                                                                                                                                                                                       
In \cite[Def. 2.1]{Aarnes:ConstructionPaper}
a solid partition of $X$ is a disjoint collection of solid sets $ \{ A_i \}_{i\in I}  \se \asx$ such that
$X = \bigsqcup_{i \in I} A_i$ and only finitely many of $A_i$'s are closed. By Lemma \ref{finiteSP} 
this definition and  Definition \ref{SsetPar} are consistent.
\end{remark}

If we take any solid set $A$ we immediately get a solid partition $\{ A,  X \sm A\}$ of $X$. Such partitions are called trivial. 
 
\begin{definition}\label{DeIrrPart}
Let  $ \{ A_i \}_{i\in I} $ be a solid partition of $X$. 
Let $I' = \{ i \in I:  \ A_i \in \ksx \}$. 
We say that $ \{ A_i \}_{i \in I} $  is an irreducible partition if 
$X \sm \bsc_{i \in J} A_i$ is connected
for any proper subset $ J \subset I'$. 
\end{definition}


\begin{example}
Let $X = \{ z \in \mathbb{C}: 1 \le |z| \le 2\}, C_1 = \{z \in X: 1 \le Re \,  z \le 2, Im \,  z = 0\}, \, C_2 = \{z \in X: -2 \le Re \, z \le -1, Im \,  z = 0\}, \, 
C_3 =  \{z \in X: 1 \le Im \,  z \le 2, Re \, z = 0\}, \, C_3 =  \{z \in X: -2 \le Im  \, z \le -1, Re \,  z = 0\}$.  
Then $C_1, C_2, C_3, C_4$ and the four components of 
$X \sm (C_1 \sc \ldots \sc C_4)$ constitute a nontrivial partition of $X$, which is not irreducible. The partition consisting of $C_1, C_2$ and the two 
components of $X \sm (C_1 \sc  C_2)$ is non-trivial irreducible. 
\end{example}

\begin{definition} \label{genus}
We define the genus of $X, \ g(X),$ to be the maximal value of $n-1$, where $n$ is the number of closed solid sets in an irreducible partition of $X$, 
where the max is taken over irreducible partitions of $X$ .
\end{definition} 

\begin{remark} \label{ReIrrPart}
(a1) If there  is only one open (closed) solid set  in a solid partition,  then there is 
only one closed  (open) solid set in this partition, and the partition is trivial. \\
(a2)
If  $ \{A_i\}_{i \in I}$ is a solid partition, and  $ |I'| = 2$ then  $ \{A_i\}_{i \in I}$ is an irreducible non-trivial partition. \\
(a3)
If  $ \{A_i\}_{i \in I}$ is a non-trivial irreducible partition then $| I \sm I'| \ge 2$. \\
(a4)
$g(X) =0$ if and only if every irreducible partition of $X$ is trivial. \\
(a5)
By \cite[Cor. 16]{Grubb:IrrPart} $g(X)+1$ is equal to the maximum number of components of a set of the form $V \cap W$ where $V$ and $W$ 
are open solid sets with $V \cup W = X$. \\
(a6)
By \cite[Cor. 15]{Grubb:IrrPart} in Definition \ref{genus} we may take open solid sets instead of closed solid sets. \\
(a7)
When $X$ is compact, connected, locally path connected $g(X) \le n(X)$ where $n(X)$ is the largest integer $n$ so that there is a surjective 
map from $ \pi_1(X)$ onto $FG(n)$. If $ \pi_1(X)$ is abelian, then $g(X) \le 1$. 
In particular, if $X$ is a topological group, $g(X) \le 1$.   See \cite{Grubb:Covering Spaces}, \cite{Grubb:IrrPart}.    
\end{remark}

\begin{remark}
We are particularly interested in the case $g(X) = 0$, which can also be described as follows: if the union of two open
solid sets in $X$ is the whole space, their intersection must be connected. 
Intuitively, $X$ does not have holes or loops.
In the case where $X$ is locally path connected, $g(X)=0$
if the fundamental group $\pi_1(X)$ is finite (in particular, if $X$ is 
simply connected). Knudsen \cite{Knudsen} was able to show that if 
$H^1(X) = 0 $ then $g(X) = 0$, and in the case of CW-complexes the converse also holds.
\end{remark}

\begin{definition} 
We say that a finite disjoint family $\{ C_j\}_{j \in J}  \se \csx$ generates a non-trivial
irreducible partition if there is $ J' \se J$ and a finite disjoint family $\{ U_i\}_{i \in I} \se \osx$ 
such that 
$\bsc_{j \in J'} C_j \sc \bsc_{i \in I} U_i =X$
is a non-trivial irreducible partition of $X$. 
\end{definition}

The next result is \cite[Lemma 2.1]{Aarnes:ConstructionPaper}.

\begin{lemma} \label{LeIrrPar}
Suppose $X \sm \bsc_{j \in J} C_j$ is disconnected  where $\{ C_j\}_{j \in J}  \se \csx$ 
is a finite disjoint family. Then   $\{ C_j\}_{j \in J}$ generates a non-trivial irreducible 
partition of $X$.  
\end{lemma} 

\begin{example}
The same finite family $\{ C_j\}_{j \in J}  \se \csx$ may generate different irreducible partitions; moreover, the number of closed solid set in these 
partitions may be different. 
Consider $X = \{  z \in \mathbb{C}: | z | \le 7 \} \sm \big( \{  z \in \mathbb{C}: | z | < 1 \}  \sc  \{  z \in \mathbb{C}: | z-5 | <1 \} \big)$. 
Let $C_1 = \{z \in X: Re \, z =0,\  1 \le Im \,  z \le 7\},  C_2 = \{z \in X: Im \, z =0, \ -7 \le Re \,  z \le -1\},  C_3= \{z \in X: Im  \, z =0, \ 1 \le Re \,  z \le 4\},  
C_4 = \{z \in X: Im \,  z =0, \ 6 \le Re  \, z \le 7\}$. Then $C_1, C_2, C_3, C_4$ generate various irreducible partitions, for instance, consisting of 
(a) $C_1, C_3, C_4$ and open solid sets that are connected components of $X \sm (C_1 \sc C_3 \sc C_4)$;
(b) $C_2, C_3, C_4$ and open solid sets that are connected components of $X \sm (C_2 \sc C_3 \sc C_4)$;
(c) $C_1, C_2$ and open solid sets that are connected components of $X \sm (C_1 \sc C_2)$.
\end{example}

\begin{corollary} \label{genconn}
$g(X) =0$  if and only if  $X \sm \bigsqcup_{i=1}^n  C_i$ is connected
for any finite disjoint family $\{C_i\}_{i=1}^n $ of closed solid sets. 
\end{corollary}

\begin{proof}
\noindent
($\Longleftarrow$ ) 
Suppose $g(X) \ge 1$. Then there exists a non-trivial irreducible
partition $\{ A_i \}_{i \in I}$ of $X$. By part (a1) of  Remark \ref{ReIrrPart}
$ X \sm \bigsqcup_{i \in I'} A_i$ is disconnected, where $I' = \{ i\in I: A_i \in \csx \}.$ 
($\Longrightarrow$ ) Suppose that $X  \sm \bigsqcup_{i=1}^n  C_i$ is disconnected.
By Lemma \ref{LeIrrPar} there is a non-trivial irreducible partition.
Then $g(X) \ge 1$ by  part (a4) of Remark \ref{ReIrrPart}. 
\end{proof} 
   
\section{Solid-set function: a different definition} \label{SSFAarnes}

When $X$ is compact we have another definition of a solid-set function:

\begin{definition} \label{ssfC}
A set function $\lambda: \asx\to\r$ is a solid-set function if
\begin{enumerate}
\item \label{superaddC}
$ \sum_{i\in I}\la(A_i)\leq\lambda(X) \ \, \mbox{if}  \ \,  \{A_i\}_{i\in I}\subseteq\asx \  \mbox{is a disjoint finite collection}. $
\item \label{regulC}
$\la(U)=\sup\{\la(C):C\in\csx,C\subseteq U\} \ \ \ \mbox{for} \ \ \ U\in\osx.$ 
\item \label{solidpartiC}
$ \sum_{i \in I} \la(A_i) = \la(X) \ \ \ \mbox{for an irreducible partition} \ \ \ \bsc_{i \in I} A_i = X.  $ \\
(In particular, $ \la(A) + \la(X \sm A) = \la(X)$ for any $ A \in \asx$.)
\end{enumerate}
\end{definition}

\begin{remark} \label{g0ssf}
An equivalent definition of a solid-set function may be obtained if one replaces the 
first condition in Definition \ref{ssfC} by the following condition:  
$ \sum_{i=1}^n \la(C_i) \le \la(C)$  if  $  \bsc_{i=1}^n C_i \se C, \ C, C_i \in \csx.$
This gives the original definition of a solid-set function from \cite{Aarnes:ConstructionPaper}.

Part \ref{regulC} in Definition \ref{ssfC} is equivalent to the following: 
$ \la(C)=\inf\{\la(U): C \se U, U \in \osx \}$ for  $C \in\csx. $

If $g(X) = 0$ then by Remark \ref{ReIrrPart} condition \ref{solidpartiC} 
in Definition \ref{ssfC} becomes:
$ \la(A) + \la(X \sm A) = \la(X)  \mbox{   for any   } A \in \asx.$
\end{remark}

\begin{lemma}\label{LePropSsfC}
For the set function $\la$ given by Definition \ref{ssfC} we have:
\begin{enumerate}
\item \label{laemp} 
$\la(\emptyset)=0.$
\item \label{lareg}
$\la(C) = \inf\{ \la(U): \ C \se U , \ U \in \osx\}$ for any $ C \in \csx.$
\item \label{lasupa}
If $ \bsc_{s \in S}  A_s \se A$ where $ A, A_s \in \asx $ then 
$\sum_{s \in S} \la(A_s) \le \la(A). $
\item \label{laAdd}
If $ \bsc_{s \in S}  A_s = A$ is a solid partition of $A$  then $\sum_{s \in S} \la(A_s) = \la(A). $
\end{enumerate}
\end{lemma}

\begin{proof} 
Parts (\ref{laemp})  and (\ref{lareg}) are easy. In part (\ref{lasupa})
it is enough to assume that $S$ is finite, since 
$\sum_{s \in S } \la(A_s) = 
\sup \{ \sum_{s \in S'} \la(A_s) : \ S' \se S,  \ S' \mbox{  is finite } \}$.
Then $(\bsc_{s \in S} A_s) \sc (X \sm A) $ is a finite disjoint collection of solid sets.
By Definition \ref{ssfC} $ \sum_{s \in S} \la(A_s)  + \la(X \sm A) \le \la(X) = \la(A) - \la(X \sm A)$, and we obtain the result.  \\
(\ref{laAdd}).
Since 
$ \bsc_{s \in S}  A_s = A$
is a solid partition of $A$ iff 
$ (X \sm A) \sc  \bsc_{s \in S}  A_s = X$ is a solid partition of $X$, 
it is enough to prove the statement for solid partitions of $X$.
Let 
$ \bsc_{j \in J} C_j \sc \bsc_{s  \in S} V_s = X, \ \ C_j \in \csx, \ V_s \in \osx$
be a solid partition  of $X$. $J$ is finite, and the proof is by induction on $|J|$. 
For $|J|=1$ or $|J|=2$  the result holds by 
Remark \ref{ReIrrPart} and Definition \ref{ssfC}. Assume that $|J| \ge 3$ and that the result 
holds for any solid partition of $X$ with less then $|J|$ closed solid sets.
Note that  $|S| \ge 2$ by Remark \ref{ReIrrPart}, i.e. $X \sm \bsc_{j \in J} C_j$ is disconnected. 
Obtain by Lemma \ref{LeIrrPar} a non-trivial irreducible partition of $X$
\begin{align} \label{st1}
 X = \bsc_{j \in J'} C_j \sc \bsc_{i \in I} U_i.
\end{align}  
Then 
$ \bsc_{i \in I} U_i = X \sm \bsc_{j \in J'} C_j = \bsc_{s \in S} V_s \sc \bsc_{j \in J \sm J'} C_j.$
By connectedness, 
for each $i \in I$
$$ U_i = \bsc_{j \in J_i} C_j \sc \bsc_{s \in S_i} V_s  \ \  \mbox{where}\   J_i = \{ j \in J \sm J': \ C_j \in U_i \}, \ S_i = \{ s \in S : \ V_s \in U_i \}. $$ 
The partition in (\ref{st1}) is non-trivial, so $|J'| \ge 2$, and then $|J_i| \le |J \sm J'| \le n-2$.
Then
$ ( X \sm U_i) \sc \bsc_{j \in J_i} C_j \sc \bsc_{s \in S_i} V_s  = X$
is a solid partition of $X$ with at most $n-1$ closed solid sets, so by induction hypothesis
$ \la(X \sm U_i) + \sum_{j \in J_i} \la(C_j)  + \sum_{s \in S_i} \la(V_s) 
= \la(X).$
By  part \ref{solidpartiC} of Definition \ref{ssfC}
\begin{eqnarray} \label{st3}
\la(U_i) =\sum_{j \in J_i} \la(C_j)  + \sum_{s \in S_i} \la(V_s).
\end{eqnarray}
Applying (\ref{st3}) and then Definition \ref{ssfC} to the irreducible partition 
in (\ref{st1})  we have:
$ \sum_{j \in J} \la(C_j) + \sum_{s \in S} \la(V_s)= 
\sum_{ j \in J'} \la(C_j) + 
\sum_{i \in I} \left( \sum_{ j \in J_i} \la(C_j)  + \sum_{s \in S_i} \la(V_s) \right) 
= \sum_{ j \in J'} \la(C_j) + \sum_{i \in I} \la(U_i) = \la(X).
$
\end{proof}

\begin{remark}
Lemma \ref{LePropSsfC} establishes the equivalence of Definition \ref{ssfC} and Definition \ref{DeSSFLC} when $X$ compact. 
\end{remark}

\section{Extension to a topological measure}  \label{Dan'sExtension}

We shall describe a way to extend a solid-set function $\la$ from Definition \ref{ssfC} to a topological measure. 
This extension procedure is related to one 
given by D. Grubb \cite{Grubb:Lectures}, which significantly simplified the original one presented in \cite{Aarnes:ConstructionPaper}.\\

\noindent
STEP 1.  We extend $\la$ to $ \la_1$ on $ \ccx \cup \assx$ $\la_1$  exactly as in formula (\ref{Xcomass}). 
Note that Lemma \ref{Prla1LC} holds for $\la_1$. \\

\noindent
STEP 2.
Let $\oox = \{ X \sm C: \ C \in \cox \}.$
We extend the set function $\la_1$ to a set function $\la_2$ defined on $\aox= \cox \cup \oox$. On $ \cox $ the function $\la_2$ is defined as in 
Definition \ref{la2LC}, and for  $ U \in \oox$  let
\begin{align} \label{UX}
 \la_2(U) = \la(X)  - \la_2(X \sm U).
\end{align}
We still have Lemma \ref{Lela2LC}, and we also have:
\begin{align} \label{Uoreg}
 \la_2(U) = \sup\{ \la_2(C) : \ C \se U, \ C \in \cox \} \ \ \ \mbox{for} \ \ U \in \oox.
\end{align}

\begin{proof}
Let $ U = \bsc_{i \in I} U_i$ be the decomposition into components of $U \in \oox$.
By Lemma \ref{DecUssN} components $U_i \in  \ossx$, and all but finitely many are in $\osx$. 
For $\epsilon>0$, choose a finite set $T\subseteq I$ for which
$\sum_{i\in I\setminus T}\la_1(U_i)<\epsilon/2$.
By part  \ref{bossEq} of Lemma \ref{Lela2LC}  for each $i\in T$  choose a set
$C_i \in \cox$ such that 
$ C_i\se U_i $ and $ \sum_{i\in T}\la_1(U_i) - \sum_{i\in T} \la_2(C_i) < \eps/2$.
Let  $C=\bsc_{i\in T} C_i \in \cox$. Then  $C \se U$ and
\begin{eqnarray} \label{vsp1}
\la_2(C)=\sum_{i\in T} \la_2(C_i)>\sum_{i\in T}\la_1(U_i)-\frac{\epsilon}{2}
\ge \sum_{i\in I}\la_1(U_i)-\epsilon.
\end{eqnarray}
Now writing 
$X\setminus U=\bsc_{j=1}^m F_j, \ \  F_j \in \ccx$ 
by part \ref{pa3} of Lemma \ref{Prla1LC} we have
\begin{align} \label{vsp2}
\la_1(X) = \sum_{j=1}^m \la_1(F_j) + \sum_{i \in I} \la_1(U_i).
\end{align}
Then by (\ref{vsp1}), (\ref{vsp2}) and (\ref{UX}) we have
$ \la_2(C) + \eps \ge \sum_{i\in I}\la_1(U_i) 
= \la_1(X)  - \sum_{j=1}^m \la_1(F_j)  = \la_2(U). $
\end{proof}

\noindent
We extend the set function $\la_2$ to a set function $\mu$ defined on $\ax$: 
for an open set $U$ define $\mu(U) $ as in Definition \ref{muLC}, and  for a closed set $C$  let 
\begin{align} \label{XCmu}
 \mu(C) = \mu(X) - \mu(X \sm C).
\end{align}

\begin{lemma} \label{PropMuC}
The set function $\mu$ satisfies the following properties:
\begin{enumerate}
\item \label{extlamb2C}
$\mu  = \la_2$ on $ \oox \cup \cox$, $\mu = \la$ on $\asx$, and $\mu$ is finite.
\item \label{monotoneC}
$\mu$ is monotone on open sets, $\mu$ is monotone on closed sets.
\item \label{regularityC}
$\mu(U) = \sup\{\mu(C): \ C \se U , \ C \in \cx \} , \ \ \ U  \in \ox$. \\
$\mu(C) = \inf\{\mu(U): \ C \se U , \ U \in \ox \} , \ \ \ C  \in \cx$.
\item\label{OpenFinAddC}
$ \mu$ is finitely additive on open sets.
\item \label{CloFinAddC}
$\mu$ is finitely additive on closed sets.
\item \label{FinAddC}
$\mu$ is finitely additive on $\ax$. 
\end{enumerate}
\end{lemma}

\begin{proof}
(1.) Note that $\mu(U) = \la_2(U)$ for any $ U \in \oox$ by (\ref{Uoreg}). For  $C \in \cox$ then
$ \mu(C) = \mu(X) - \mu(X \sm C) = \la_2(X) - \la_2(X \sm C) = \la_2(C).$
Thus, $\mu = \la_2 $ on $ \oox \cup \cox$. In particular, $\mu = \la_2 = \la$ on $\csx$, 
and hence, $\mu = \la$ on solid sets; $\mu(X) = \la(X) < \infty$. 
(2.) Obvious. 
(3.) As in the proof of part \ref{regularityLC} of Lemma \ref{PropMuLC} we have
$\mu(U) = \sup\{\mu(K): \ K \se U , \ K \in \cx \}$. 
The second statement follows immediately from the first by definition of $\mu$. 
(4.) and (5.) are proved as in Lemma \ref{PropMuLC}. 
(6.) Suppose $A,B\in\ax$ are disjoint. By parts (\ref{OpenFinAddC}) and 
(\ref{CloFinAddC})  we may assume that $A, B$ are not both closed
or both open. We may also assume that $A$ and $A\sqcup B$ are not
both closed or both open. Then the sets
$A$ and $X\setminus (A\sqcup B)$ are both closed or both open, and disjoint.
Hence by part (\ref{OpenFinAddC}) or part (\ref{CloFinAddC}) 
$ \mu(A)+\mu(X\setminus (A\sqcup B))=\mu \left(A\sqcup  [X\setminus (A\sqcup B)] \right) =
\mu(X\setminus B), $
which by (\ref{XCmu}) gives
$ \mu(A)+\mu(X) - \mu(A \sc B)=\mu (X) - \mu(B)$, 
i.e.
$ \mu(A) + \mu(B) = \mu(A \sc B).$
\end{proof}

\begin{theorem}\label{ExtThC} 
Any solid-set function on $X$ extends uniquely to a topological measure on $X$.
\end{theorem}

\begin{proof}
Parts (\ref{regularityC}), (\ref{FinAddC}) and (\ref{extlamb2C}) of Lemma \ref{PropMuC}
show that that the set function $\mu$ is a topological  
measure that is an extension of a given solid-set function $\la$. 
The proof of uniqueness is similar to the one in Theorem \ref{extThLC}. 
\end{proof} 

\section{Examples for a compact space} \label{ExamplesC}

In the following examples the 
underlying compact space has genus 0. This means by Remark \ref{g0ssf}  that checking
the last condition of a solid-set function in Definition \ref{ssfC} is greatly simplified. 

\begin{example}[Aarnes circle measure] \label{Aatm}
Let $X$ be the unit circle and $B$ be the boundary of $X.$
Fix a point $p$ in $X \setminus B$.
Define $\mu $ on solid sets as follows:
$\mu (A) = 1$ if i) $B \subset A$ or
ii) $ p \in A $ and $A \cap B \ne  \O$.
Otherwise, we let $ \mu(A) = 0 $.
Then 
$ \mu $ is a solid-set function and hence
extends to a topological measure on $X$.
Note that $\mu$ is not a point mass. 
To demonstrate that $\mu $ is 
not a measure we shall show that $\mu$ is not
subadditive. 
Let $A_1$ be a closed solid set which is an arc that is a proper subset of $B$, 
$A_2$ be a closed solid set that is the closure of $B \sm A_1$, 
and $A_3 = X \setminus B$ be
an open solid subset of $X$. Then 
$X =  A_1 \cup A_2 \cup A_3, \  \mu(X) = 1 $, but 
$ \  \mu (A_1) + \mu(A_2) + \mu(A_3) = 0$.
\end{example}

\begin{example} \label{dscrAatm}
Another solid-set function (hence, a topological measure) is obtained if  in the 
Example \ref{Aatm} we take  $B$ to be any set of points of $X$, and $ p \in X \sm B$.
\end{example} 

\begin{example} \label{3pts2vtm}
Let $X$ be a sphere (or a square.)
Fix points $p_1, p_2, p_3$  in $X$.
Define $\mu $ on solid sets as follows:
$\mu (A) = 1$ if $A$ contains the majority of the three points, 
otherwise, $\mu(A) = 0$.
The resulting topological measure is non-subadditive,
since $ \mu (X) =1 $, and it is easy to
represent $X$ as a union of three overlapping
solid sets each of which contains exactly
one of the points $x,y,z$ and, hence, has measure 0.
Notice also that $\mu(A) =1$ for any connected set $A$ that contains at least 2 points, 
since any component of $X \sm A$ (a solid set by Lemma \ref{SolidCompoLC}) contains at most 1 point. 
\end{example}

\begin{example} \label{npts2vtm}
Let $X$ be a sphere
and let $ P= \{ p_1, \ldots, p_n \} $ be a set of $n$ distinct points in $X$, with $n$ an 
odd number. We define a solid-set function $\mu$ on $X$ by letting $\mu (A) = 1$ if $ A \cap P$
contains a majority of points of $P$; otherwise let $\mu (A) = 0.$ 
\end{example}

\begin{example}  \label{n+1vtm}
Let $X$ be the unit sphere and $ P= \{ p_1, \ldots, p_{2n+1} \} $ an odd-numbered subset of $X$.
If $ A \in \asx$ let $\sharp A =$ the number of points in $ A \cap P.$ For $ k = 0, \ldots, n$
let $ \mu (A) = k/n$ if $ \sharp A \in \{ 2k, 2k+1\}.$ Thus defined,  
$\mu$ is a solid-set function in $X.$
\end{example}

\begin{remark}
Example \ref{Aatm}, Example \ref{3pts2vtm}, Example \ref{n+1vtm} were the first examples of topological measures that are not 
(restrictions to $\ox \cup \cx$ of) regular Borel measures. They were presented by J. Aarnes in \cite{Aarnes:TheFirstPaper}, 
\cite{Aarnes:ConstructionPaper} and \cite{Aarnes:Pure}.  For more examples of topological measures on compact spaces see, for instance, 
\cite{AarnesRustad} and \cite{QfunctionsEtm}.
\end{remark}

\section{Examples for a locally compact space} \label{ExamplesLC}

When $X$ is locally compact,  the hardest condition in Definition  \ref{DeSSFLC} 
to verify is the condition \ref{solidparti} that deals with solid partitions. 
But, as we shall see in this section, it turns out that this condition holds trivially for spaces 
whose one-point compactification has genus $0$.
We denote by $\hat X$ the  one-point compactification of $X$.

\begin{lemma} \label{hatXsoli}
Let $X$ be locally compact noncompact and $\hat X$ be its one-point compactification.
If $A \in \mathscr{A}_{s}^{*}(X)$ then $A$ is solid in $\hat X$.
\end{lemma}

\begin{proof}
Since $A $ is connected in $X$, it is also is connected in $\hat X$.
Let $X \setminus A = \bigsqcup_{i=1}^n B_i$ be the decomposition into connected components.
Each $B_i$ is an unbounded subset of $X$. 
We can write $\hat X \setminus A = \bigcup_{i=1}^n E_i$ where each $E_i = B_i \cup \{ \infty\}$.
It is easy to see that each $E_i$ is connected in $\hat X$.  Thus,  $\hat X \setminus A$ is connected,
and so  $A$ is solid in $\hat X$.
\end{proof}

\begin{lemma} \label{nosopart}
Let $X$ be a locally compact noncompact  space whose one-point compactification $\hat X$ has genus 0. 
If $A \in \mathscr{A}_{s}^{*}(X)$ then any solid partition of $A$ is the set $A$ itself.
\end{lemma}

\begin{proof} 
Suppose first that $V \in \mathscr{O}_{s}^{*}(X)$ and its solid partition is given by  
$ V = \bigsqcup_{i=1}^n C_i \sqcup \bigsqcup_{j=1}^m U_j $, 
where each $C_i \in \mathscr{K}_{s}(X)$ and each $U_j \in \mathscr{O}_{s}^{*}(X)$. From Lemma \ref{hatXsoli}
it follows that $\hat X \setminus V$ and each $C_i$ are closed solid sets in $\hat X$.
Since $\hat X$ has genus $0$, by Remark \ref{genconn},  
$ \hat X \setminus ((\hat X \setminus V) \sqcup \bigsqcup_{i=1}^n C_i)  = \bigsqcup_{j=1}^m U_j $
must be connected in $\hat X$. It follows that $m=1$ and we may write 
$V = \bigsqcup_{i=1}^n C_i \sqcup  U_1. $
Then $\{ U_1, \hat X \setminus V, C_1, \ldots, C_n\} $ is a solid partition of $\hat X$, and it has 
only one open set. By Remark \ref{ReIrrPart} this solid partition also has only one 
closed set in it,  and it must be $\hat X \setminus V$. So each $C_i= \emptyset$, and the solid partition of $V$ is
$V = U_1$, i.e. the set itself.

Now suppose that $C \in \mathscr{K}_{s}(X)$ and its solid partition is given by  
$ C = \bigsqcup_{i=1}^n C_i \sqcup \bigsqcup_{j=1}^m U_j $,
where each $C_i \in \mathscr{K}_{s}(X)$ and each $U_j \in \mathscr{O}_{s}^{*}(X)$. 
Then $\{ \hat X \setminus C, U_1, \ldots, U_m, C_1, \ldots, C_n\}$ is a solid partition of $\hat X$.
Again by Remark \ref{genconn},
$ \hat X \setminus \bigsqcup_{i=1}^n C_i =  (\hat X \setminus C) \sqcup U_1 \ldots \sqcup U_m $
must be connected in $\hat X$. It follows that $U_j = \emptyset$ for $j=1, \ldots, m$.
Then by connectedness of $C$ we see that the solid partition of $C$ must be the set itself.
\end{proof}

\begin{remark} \label{easyRn} 
From Lemma \ref{nosopart}
it follows that for any locally compact  noncompact  space whose one-point compactification has genus 0 
the last condition of Definition \ref{DeSSFLC} holds trivially. This is true, for example, 
for $X = \mathbb{R}^n $, half-plane in $\mathbb{R}^n$ with $n \ge 2$, or for a punctured ball in $\mathbb{R}^n$ with the relative topology. 
\end{remark}

\begin{example}
Lemma \ref{nosopart} may not be true for spaces whose  one-point compactification 
has genus greater than $0$.  For example, let $X$ be an infinite strip  $\mathbb{R} \times [0,1]$ without the ball 
of radius $1/4$ centered at $(-1/2, 1/2)$,
so  $\hat X$ has genus greater than $0$.   
It is easy to give an example of  a solid partition of a bounded solid set 
(say, rectangle $[0,n] \times [0,1]$ or $(0,n) \times [0,1]$) which consists of $n$ solid sets 
(rectangles of the type $(i, i+1) \times [0,1]$ or  $[i, i+1] \times [0,1]$) for any given odd $n \in \mathbb{N}, \, n >1$.
\end{example}

We are ready to give examples of topological measures on locally compact spaces. 

\begin{example} \label{ExDan2pt}
Let $X$ be a locally compact space whose one-point compactification has genus 0. 
Let $\lambda$ be a real-valued topological measure on $X$ 
(or, more generally, a real-valued deficient topological measure on $X$; 
for definition and properties of deficient topological measures on locally compact spaces see \cite{Butler:DTMLC}). 
Let $P$ be a set of two distinct points. 
For each $A  \in \mathscr{A}_{s}^{*}(X)$ let $ \nu(A) = 0$ if $\sharp A = 0$,  $ \nu(A) = \lambda(A) $ if $\sharp A = 1$, and 
$ \nu(A) = 2 \lambda(X)$ if $\sharp A = 2$,
where $ \sharp A$ is the number of points in $ A \cap P$.
We claim that $\nu$ is a solid-set function. 
By Remark \ref{easyRn} we only need to check the first three conditions of 
Definition \ref{DeSSFLC}. The first one is easy to see. 
Using Lemma \ref{LeCsInside} it is easy to verify conditions \ref{regul} and \ref{regulo} 
of Definition \ref{DeSSFLC}. 
The solid-set function $\nu$ extends to a unique compact-finite topological measure on $X$. 
Suppose, for example, that  $ \lambda$ is the Lebesgue measure on $X = \mathbb{R}^2$, the set $P$ consists of two points
$p_1 = (0,0)$ and $p_2 = (2,0)$. Let $K_i$ be the closed ball of radius $1$ centered at $p_i$ for $i=1,2$. Then 
$K_1, K_2$ and $ C= K_1 \cup K_2$ are compact solid sets, $\nu(K_1) = \nu(K_2) = \pi, \,  \nu(C) = 4 \pi$. Since 
$\nu$ is not subadditive, $\nu$ is a topological measure that is not a measure. Note that $\nu(X) = \infty$. 
\end{example}

\noindent
The next two examples are adapted from \cite[Example 2.2]{Aarnes:LC} and are related to Example \ref{Aatm}.

\begin{example} \label{puncdisk}
Let $X$ be the unit disk on the plane with removed origin. $X$ 
is a locally compact  Hausdorff space with respect to the relative topology. 
Any subset of $X$ whose closure in $\mathbb{R}^2$ contains the origin is unbounded in $X$.
For $A \in \mathscr{A}_{s}^{*}(X)$ (since $A$ is also a solid subset of the unit disk by Lemma \ref{hatXsoli})  we define 
$\mu' (A) = \mu(A)$ where $\mu$ is the solid-set function on the unit disk 
from Example \ref{Aatm}.  From Remark \ref{easyRn}, Lemma \ref{LeCsInside}, 
and the fact that $\mu$ is a solid-set function
on $\hat X$  we see that $\mu'$ is a solid-set function on $X$. By Theorem \ref{extThLC}
$\mu'$ extends uniquely to a topological measure on $X$, which we also call $\mu'$.  
Note that $\mu'$ is simple. We claim that $\mu'$ 
is not a measure.  Let $U_1 = \{ z \in X:  \ Im \ z > 0\}, \ U_2 =   \{ z \in X:  \ Im \ z < 0\}$ and 
$F =  \{ z \in X:  \ Im \ z = 0\}$. Then $U_1, U_2$ 
are open (unbounded) in $X$  and $F$ is a closed (unbounded) set in $X$ consisting of 
two disjoint segments.  Note that $X= F \cup U_1 \cup U_2$. 
Using Remark \ref{extsumme} we calculate $\mu'(F) = \mu'(U_1) = \mu'(U_2) =0$.
The boundary of the disk, $C$, is a compact connected set, $X \setminus C$ is unbounded in $X$, so $C \in \mathscr{K}_{s}(X)$. 
Since $\mu'(C) = 1$, we have $\mu'(X) =1$. Thus, $\mu'$ is not subadditive, so it is not a measure. 

This example also shows that on the unit disk without the origin (a locally compact space)
finite additivity of topological measures holds on $\mathscr{K}(X) \cup \mathscr{O}(X)$ by 
Definition \ref{TMLC}, but fails on $\mathscr{C}(X) \cup \mathscr{O}(X)$. This is in contrast to topological measures on the unit disk
(and on all compact spaces), where finite additivity holds on $\mathscr{C}(X) \cup \mathscr{O}(X)$. 
\end{example}

\begin{example} \label{linetm}
Let $X = \mathbb{R}^2, \ l$ be a straight line and $p$ a point of $X$ not on the line $l$. 
For $A \in \mathscr{A}_{s}^{*}(X)$  define $\mu(A) = 1$ if $A \cap l \neq \emptyset$ and $p \in A$; otherwise,
let $\mu(A) =0$.  Using Lemma \ref{LeCsInside} 
it is easy to verify the first three conditions of Definition \ref{DeSSFLC}.
From Remark \ref{easyRn} it follows that $\mu$ is a solid-set function on $X$. 
By Theorem \ref{extThLC} $\mu$ extends uniquely to a topological measure on $X$, 
which we also call $\mu$.  Note that $\mu$ is simple. We claim that $\mu$ is not a measure. 
Let $F$ be the  closed half-plane determined by $l$ which does not contain $p$.  
Then using Remark \ref{extsumme} we calculate $\mu(F) = \mu(X \setminus F) = 0$, 
and $\mu(X) = 1$. Failure of subadditivity shows that $\mu$ is not a measure.

The sets $F$ and $X \setminus F$ are both unbounded.  Now take a bounded open disk 
$V$ around  $p$ that does not intersect $l$. Then 
$ X = V \sqcup (X \setminus V), $
where $V \in \mathscr{O}^{*}(X), \ \mu(V) = \mu(X \setminus V) = 0$, while $\mu(X) =1$.

This example also shows that on a locally compact space
finite additivity of topological measures holds on $\mathscr{K}(X) \cup \mathscr{O}(X)$ by 
Definition \ref{TMLC}, but fails on $\mathscr{C}(X) \cup \mathscr{O}(X)$. 
It fails even in the situation $X = V \sqcup F$, where $ V $ is a bounded open set, and $F$ is a closed set.
\end{example}

The last two examples suggest 
the following result.

\begin{theorem} \label{tmXtoXha}
Let $X$ be a noncompact  locally compact, connected, locally connected space whose one-point compactification $\hat X$ has genus 0. 
Suppose $\nu$ is a solid-set function on $\hat X$.  For $A \in \mathscr{A}_{s}^{*}(X)$  define $\mu(A) = \nu(A)$. 
Then $\mu$ is a solid-set function on $X$  and, thus, extends uniquely to a topological 
measure on $X$.
\end{theorem}

\begin{proof}
Let $A \in \mathscr{A}_{s}^{*}(X)$. By Lemma \ref{hatXsoli},  $A$ is a solid set in $\hat X$. 
Using Lemma \ref{LeCsInside}, the fact that $\nu$ is a solid-set function on $\hat X$, and that 
a bounded solid set does not contain $\infty$ it is easy to verify the first three conditions of Definition \ref{DeSSFLC}.
By Remark \ref{easyRn} $\mu$ is a solid-set function on $X$. 
\end{proof}

Theorem \ref{tmXtoXha} allows us to obtain a large variety of topological measures on a locally compact space from
examples of topological measures on compact spaces. 

\begin{example} \label{nvssf}
Let $X$ be a locally compact space whose one-point compactification has genus 0. 
Let $n$ be a natural number. Let $P$ be the set of distinct $2n+1$ points.
For each $A  \in \mathscr{A}_{s}^{*}(X)$ let $ \nu(A) = i/n$ if $\sharp A = 2i$ or $2i+1$, where $ \sharp A$ is the number of points in $ A \cap P$.
By Example \ref{n+1vtm}  and Theorem \ref{tmXtoXha}  $\nu$ is a solid-set function on $X$; it extends to a unique topological measure on $X$ 
that assumes values $0, 1/n, \ldots, 1$. 
\end{example}

We conclude with an example of another  infinite topological measure.

\begin{example} \label{mojexLC}
Let $X=\mathbb{R}^n$ for any $n \ge 2$, and $\lambda$ be the Lebesque measure on $X$.
For $U \in \mathscr{O}_{s}^{*}(X)$ define $\mu(U) =0$ if $0 \le \lambda(U) \le 1$ and $\mu(U) = \lambda(U)$ if
$\lambda(U) >1$. For $C \in \mathscr{K}_{s}(X)$ define $\mu(C) = 0$ if $0 \le \lambda(C) < 1$ and 
$\mu(C) =\lambda(C)$ if $\lambda(C) \ge 1$.  It is not hard to check the first three conditions 
of Definition \ref{DeSSFLC}. From Remark \ref{easyRn} it follows that $\mu$ is 
a solid-set function on $X$. 
By Theorem \ref{extThLC} $\mu$ extends uniquely to a topological measure on $X$, 
which we also call $\mu$.  Note that $\mu(X) = \infty$. $\mu$ is not subadditive, for  
we may cover a compact ball with Lebesque measure greater than 1 by finitely many balls 
of Lebesque measure less than 1. Hence, $\mu$ is not a measure.
\end{example}

The vast majority of papers dealing with quasi-linear functionals and topological measures (including papers in symplectic geometry) 
consider a compact underlying space, finite topological measures and bounded quasi-linear functionals.
In symplectic geometry one often sees quasi-linear functionals corresponding to simple topological measures. 
This section shows how to obtain a variety of finite and infinite topological measures on locally compact spaces.

{\bf{Acknowledgments}}:
This  work was conducted at 
the Department of Mathematics at the University of California Santa Barbara. The author would like to thank the department
for its hospitality and supportive environment.


\end{document}